\documentclass{article}

\usepackage{amsmath,amssymb,amsthm,mathrsfs}
\usepackage{hyperref,xcolor}
\usepackage{changepage}
\usepackage{enumitem} 
\usepackage{graphicx}
\usepackage{tikz}
\usepackage[utf8]{inputenc}
\usepackage[T1]{fontenc}
\usepackage{hyperref}
\usepackage{comment}

\allowdisplaybreaks

\numberwithin{equation}{section}

\newtheorem{myDefn}{Definition}[section]

\newtheorem{myProp}[myDefn]{Proposition}

\newtheorem{myRem}[myDefn]{Remark}

\newtheorem{myExa}[myDefn]{Example}

\newtheorem{myLem}[myDefn]{Lemma}

\newtheorem{myTheorem}[myDefn]{Theorem}

\DeclareMathOperator*{\argmin}{argmin}

\def\nn{\mathrm{n}}

\usepackage{indentfirst}

\def\R{\mathbb{R}}

\def\HH{\mathrm{H}}
\def\LL{\mathrm{L}}

\def\MD{\mathcal{D}}

\newcommand{\fonction}[5]{\begin{array}[t]{lrcl}#1 :&#2 &\longrightarrow &#3\\&#4& \longmapsto &#5 \end{array}}

\newcommand{\dual}[2]{\left\langle #1 , #2 \right\rangle}
\newcommand{\hookdoubleheadrightarrow}{%
  \hookrightarrow\mathrel{\mspace{-15mu}}\rightarrow
}

\newlist{primenumerate}{enumerate}{1}
\setlist[primenumerate,1]{label={\roman*$'$}}

\title{Sensitivity Analysis of a Scalar Mechanical Contact Problem with Perturbation of the Tresca’s Friction Law}

\author{Lo\"ic Bourdin\footnote{Institut de recherche XLIM. UMR CNRS 7252. Universit\'e de Limoges, France. \texttt{loic.bourdin@unilim.fr}}, 
Fabien Caubet\footnote{Universit\'e de Pau et des Pays de l'Adour, E2S UPPA, CNRS, LMAP, UMR 5142, 64000 Pau, France. \texttt{fabien.caubet@univ-pau.fr}}, Aymeric Jacob de Cordemoy\footnote{Universit\'e de Pau et des Pays de l'Adour, E2S UPPA, CNRS, LMAP, UMR 5142, 64000 Pau, France. \texttt{aymeric.jacob-de-cordemoy@univ-pau.fr}}
}

\begin{document}

\maketitle

\begin{abstract}
This paper investigates the sensitivity analysis of a scalar mechanical contact problem described by a boundary value problem involving the Tresca's friction law. The sensitivity analysis is performed with respect to right-hand source and boundary terms perturbations. In particular the friction threshold involved in the Tresca's friction law is perturbed, which constitutes the main novelty of the present work with respect to the existing literature. Hence we introduce a parameterized Tresca friction problem and its solution is characterized by using the proximal operator associated with the corresponding perturbed nonsmooth convex Tresca friction functional. Then, by invoking the extended notion of twice epi-differentiability depending on a parameter, we prove the differentiability of the solution to the parameterized Tresca friction problem, characterizing its derivative as the solution to a boundary value problem involving Signorini unilateral conditions. Finally numerical simulations are provided in order to illustrate our main result.
\end{abstract}

\bigskip

\textbf{Keywords:} mechanical contact problems, Tresca’s friction law, Signorini unilateral conditions, variational inequalities, convex subdifferential, proximal operator, sensitivity analysis, twice epi-differentiability.  

\medskip

\textbf{AMS Classification:} 49Q12, 46N10, 74M15.

\section{Introduction}
\paragraph{Mechanical context and motivations.} On the one hand, mathematical models for mechanical contact problems between deformable bodies are investigated in the literature in view of various engineering applications, such as the analysis of the wheel-ground contact for a vehicle, the study of the contact of a rocket structure with the atmosphere, etc. \textit{Contact mechanics} describes the deformation of solids that touch each other on parts of their boundaries. Mostly, the mechanical setting consists in a deformable body which is in contact with a rigid foundation without penetrating it and possibly sliding against it which causes friction. From the mathematical point of view, these phenomena translate into different constraints: the non-permeability conditions take the form of inequalities on the contact surface called {\it Signorini unilateral conditions} (see, e.g.,~\cite{15SIG,16SIG}); the friction occurring on the contact surface is typically modeled by the so-called {\it Tresca’s friction law} (see, e.g.,~\cite{KUSS}) which appears as a boundary condition involving nonsmooth inequalities depending on a friction threshold. Finally those mechanical contact problems are usually investigated through the theory of {\it variational inequalities}, and Signorini unilateral conditions and the Tresca's friction law cause nonlinearities and/or nonsmoothness in the corresponding variational formulations.

On the other hand, \textit{shape optimization} is the mathematical field aimed at finding the optimal shape of a given object for a given criterion, that is the shape which minimizes a certain cost functional while satisfying given constraints. In order to numerically solve a shape optimization problem, the standard gradient descent method requires to compute the {\it shape gradient} of the cost functional which usually depends on the solution to a partial differential equation with given boundary conditions. Therefore a first crucial point in numerical shape optimization is to perform the {\it sensitivity analysis} of the boundary value problem with respect to perturbations.

Naturally, mechanical contact problems are ubiquitous in shape optimization and increasingly popular in industry in order to identify the optimal design of a product, like for instance the optimal shape of a structure with the maximum stiffness or the minimum weight. The present work was initially motivated by shape optimization problems involving mechanical contact and friction phenomena. For this purpose, the objective of the present paper is to investigate the sensitivity analysis of a boundary value problem involving the Tresca's friction law with respect to right-hand source and boundary terms perturbations. Especially, the friction threshold associated with the Tresca's friction law is perturbed which is the crucial point and the main novelty of the present paper with respect to the existing literature. In this work we focus on the scalar version of the Tresca's friction law, constituting a first step towards the non-trivial adaptation to the elasticity case which will be the topic of future investigations.

\paragraph{Objectives and methodology.}
The sensitivity analysis of some mechanical contact problems has already been investigated in the literature. For example, the sensitivity analysis of some friction problems are studied using the notion of conical differentiability in~\cite{SOKOL2}, or with a regularization procedure in~\cite{HAN}. The present paper follows from the previous work~\cite{4ABC} in which a novel approach, based on the notion of twice epi-differentiability introduced by R.T.\ Rockafellar in 1985 (see~\cite{Rockafellar}), has been developed. Precisely we focus on the \textit{parameterized Tresca friction problem} given by
\begin{equation}\label{justeintro}\tag{TP$_{t}$}
\arraycolsep=2pt
\left\{
\begin{array}{rcll}
-\Delta u_{t} &=&f_{t}    & \text{ in } \Omega , \\
u_{t}	 &=&  0  & \text{ on } \Gamma_{\mathrm{D}} ,\\
\partial_{\nn} u_{t} &=&k_{t} 	  & \text{ on } \Gamma_{\mathrm{N}} ,\\
|\partial_{\nn}u_{t}|\leq g_{t} \text{ and }  u_{t}\partial_{\nn}u_{t}+g_{t}|u_{t}|&=&0  & \text{ on } \Gamma_{\mathrm{T}},
\end{array}
\right.
\end{equation}
for all $t\geq0$, where $\Omega \subset \R^{d}$ is a nonempty bounded connected open subset of $\R^{d}$, $d\geq1$, with a Lipschitz continuous boundary $\Gamma:=\partial\Omega$. We assume that the boundary is decomposed as follows~$\Gamma=:\Gamma_{\mathrm{D}}\cup\Gamma_{\mathrm{N}}\cup\Gamma_{\mathrm{T}}$, where $\Gamma_{\mathrm{D}}$, $\Gamma_{\mathrm{N}}$, $\Gamma_{\mathrm{T}}$ are three measurable pairwise disjoint subsets of $\Gamma$ such that $\Gamma_{\mathrm{D}}$ and $\Gamma_{\mathrm{T}}$ have a positive measure and almost every point of $\Gamma_{\mathrm{T}}$ is an interior point (see Remark~\ref{rmqhyptresca} for details). Moreover we assume that $f_{t}\in\LL^{2}(\Omega)$, $k_{t}\in\LL^{2}(\Gamma_{\mathrm{N}})$, and $g_{t}\in\LL^{2}(\Gamma_{\mathrm{T}})$ with $g_{t}>0$ almost everywhere (a.e.) on $\Gamma_{\mathrm{T}}$, for all $t\geq0$. The boundary condition that appears on $\Gamma_{\mathrm{T}}$ corresponds to the scalar version of the Tresca's friction law (see, e.g.,~\cite[Section~5 Chapter~2]{GLOW1} or~\cite[Section~1.3 Chapter~1]{GLOW2}). The main difference with the previous paper~\cite{4ABC} is that the friction threshold in the Tresca's friction law, denoted by $g_{t}$, depends on the parameter~$t \geq 0$ (while~$g_t = g$ does not in~\cite{4ABC}). Although this change may seem innocent, we emphasize that this work is not a simple replica of the previous paper~\cite{4ABC}. Indeed this novelty implies several non-trivial technical adjustments. In particular it requires an extended version of the notion of twice epi-differentiability, as explained below.

The main purpose of this work is to characterize the derivative of the map $t\in\mathbb{R}^{+}\mapsto u_{t}\in\HH^{1}_{\mathrm{D}}(\Omega)$ at $t=0$, where
$$
 \HH^{1}_{\mathrm{D}}(\Omega):=\left\{ \varphi\in\HH^{1}(\Omega) \text{, } \varphi=0 \text{ almost everywhere on } \Gamma_{\mathrm{D}}\right\}.
$$
For $t\geq0$ fixed, the main difficulty in the analysis of the parameterized Tresca friction problem~\eqref{justeintro} comes from the absolute value map $\left|\cdot\right|$ on the boundary~$\Gamma_{\mathrm{T}}$ which generates nonsmooth terms in the corresponding variational formulation, that is to find~$u_{t}\in\HH^{1}_{\mathrm{D}}(\Omega)$ such that
$$
     \int_{\Omega} \nabla u_{t}\cdot\nabla (v-u_{t})+\int_{\Gamma_{\mathrm{T}}}g_{t}|v|-\int_{\Gamma_{\mathrm{T}}}g_{t}|u_{t}|\geq \int_{\Omega}f_{t}(v-u_{t})+\int_{\Gamma_{\mathrm{N}}}k_{t}(v-u_{t}), \qquad \forall v\in\HH^{1}_{\mathrm{D}}(\Omega).
$$
Defining the \textit{parameterized Tresca friction functional} by
$$
\fonction{\Phi}{\mathbb{R}^{+}\times \HH^{1}_{\mathrm{D}}(\Omega)}{\R}{(t,w)}{\displaystyle \displaystyle \Phi(t,w):=\int_{\Gamma_{\mathrm{T}}}g_{t}|w|,}
$$
and using the proximal operator (see Definition~\ref{proxi}) introduced by J.-J. Moreau in 1965 (see~\cite{MOR}), the solution to the parameterized Tresca friction problem~\eqref{justeintro} is characterized by
$$
\displaystyle u_{t}=\mathrm{prox}_{\Phi(t,\cdot)}(F_{t}),
$$
where $F_{t}\in\HH^{1}_{\mathrm{D}}(\Omega)$ is the unique solution to the classical \textit{parameterized Dirichlet-Neumann problem} given by
$$
 \int_{\Omega}\nabla F_{t}\cdot\nabla \varphi=\int_{\Omega}f_{t}\varphi+\int_{\Gamma_{\mathrm{N}}}k_{t}\varphi, \qquad\forall\varphi\in\HH^{1}_{\mathrm{D}}(\Omega),
$$
for all~$t \geq 0$. 
For all~$t \geq 0$, note that $\Phi(t,\cdot)$ is a lower semi-continuous convex proper function on~$\HH^{1}_{\mathrm{D}}(\Omega)$ and thus $\mathrm{prox}_{\Phi(t,\cdot)}$ is well-defined. 
Then, taking into account of the above characterization of~$u_{t}$, the differentiability of the map $t\in\mathbb{R}^{+}\mapsto u_{t}\in\HH^{1}_{\mathrm{D}}(\Omega)$ at $t=0$ is strongly related to the differentiability (in a generalized sense) of the proximal operator $\mathrm{prox}_{\Phi(t,\cdot)}$. In the previous paper~\cite{4ABC}, since the friction threshold~$g_t = g$ is not perturbed, then $\Phi(t,\cdot) = \Phi$ and $u_t =\mathrm{prox}_{\Phi}(F_t)$, and thus the standard notion of twice epi-differentiability of~$\Phi$ is used in order to derive the differentiability of the map $t\in\mathbb{R}^{+}\mapsto u_{t}\in\HH^{1}_{\mathrm{D}}(\Omega)$ at $t=0$. In the present work, since the friction threshold~$g_t$ is perturbed, and thus $u_{t}=\mathrm{prox}_{\Phi(t,\cdot)}(F_{t})$, we are driven to use an extended version of the twice epi-differentiability of~$\Phi(t,\cdot)$ (depending on a parameter) introduced in the recent paper~\cite{8AB} (see also Definition~\ref{epidiffpara}).

\paragraph{Main result.}
With the previous methodology and under some appropriate assumptions described in Theorem~\ref{caractu0derivDNT}, we prove that the map $t\in\mathbb{R}^{+}\mapsto u_{t}\in\HH^{1}_{\mathrm{D}}(\Omega)$ is differentiable at $t=0$, and its derivative $u'_{0}\in\HH^{1}_{\mathrm{D}}(\Omega)$ is given by
$$
\displaystyle u_{0}'=\mathrm{prox}_{\mathrm{D}_{e}^{2}\Phi(u_{0}|F_{0}-u_{0})}(F_{0}'),
$$
where $\mathrm{D}_{e}^{2}\Phi(u_{0}|F_{0}-u_{0})$ stands for the second-order epi-derivative (see Definition~\ref{epidiffpara}) of the parameterized Tresca friction functional~$\Phi$ at~$u_{0}$ for $F_{0}-u_{0}$, and where $F'_{0}\in\HH^{1}_{\mathrm{D}}(\Omega)$ is the derivative at~$t=0$ of the map $t\in\mathbb{R}^{+}\mapsto F_{t}\in \HH^{1}_{\mathrm{D}}(\Omega)$. Moreover we prove that $u'_{0}\in\HH^{1}_{\mathrm{D}}(\Omega)$ exactly corresponds to the unique weak solution to the \textit{Signorini problem}
\begin{equation}\tag{SP$_{0}'$}
\arraycolsep=2pt
\left\{
\begin{array}{rcll}
-\Delta u_{0}' 	& = &  f_{0}'  & \text{ in } \Omega , \\
u_{0}'	& = & 0  & \text{ on } \Gamma_{\mathrm{D}}\cup\Gamma^{u_{0},g_{0}}_{\mathrm{T}_{\mathrm{S}_{\mathrm{D}}}}, \\
\partial_{\nn} u_{0}' & = & k_{0}'  & \text{ on } \Gamma_{\mathrm{N}}, \\
\partial_{\nn} u_{0}' & = & g_{0}'\frac{\partial_{\nn}u_{0}}{g_{0}}  & \text{ on } \Gamma^{u_{0},g_{0}}_{\mathrm{T}_{\mathrm{S}_{\mathrm{N}}}}, \\
 u_{0}'\leq0\text{, } \partial_{\nn}u_{0}'\leq g_{0}'\frac{\partial_{\nn}u_{0}}{g_{0}} \text{ and } u_{0}'\left(\partial_{\nn}u_{0}'-g_{0}'\frac{\partial_{\nn}u_{0}}{g_{0}}\right) & = & 0 & \text{ on } \Gamma^{u_{0},g_{0}}_{\mathrm{T}_{\mathrm{S-}}}, \\
 u_{0}'\geq0\text{, } \partial_{\nn}u_{0}'\geq g_{0}'\frac{\partial_{\nn}u_{0}}{g_{0}} \text{ and } u_{0}'\left(\partial_{\nn}u_{0}'-g_{0}'\frac{\partial_{\nn}u_{0}}{g_{0}}\right) & = & 0  & \text{ on } \Gamma^{u_{0},g_{0}}_{\mathrm{T}_{\mathrm{S+}}}.
\end{array}
\right.
\end{equation}
The subdivision $\Gamma_{\mathrm{T}}=\Gamma^{u_{0},g_{0}}_{\mathrm{T}_{\mathrm{S}_{\mathrm{N}}}}\cup
\Gamma^{u_{0},g_{0}}_{\mathrm{T}_{\mathrm{S}_{\mathrm{D}}}}\cup\Gamma^{u_{0},g_{0}}_{\mathrm{T}_{\mathrm{S-}}}\cup\Gamma^{u_{0},g_{0}}_{\mathrm{T}_{\mathrm{S+}}}$ is described in details in Theorem~\ref{caractu0derivDNT}. Here~$f'_{0}\in\LL^{2}(\Omega)$ (resp.\ $k'_{0}\in~\LL^{2}(\Gamma_{\mathrm{N}})$) is the derivative at $t=0$ of the map $t\in\mathbb{R}^{+}\mapsto f_{t}\in \mathrm{L}^{2}(\Omega)$ (resp.\ $t\in\mathbb{R}^{+}\mapsto k_{t}\in \mathrm{L}^{2}(\Gamma_{\mathrm{N}})$) and~$g'_{0}\in\LL^{2}(\Gamma_{\mathrm{T}})$ is the map defined for almost every $s\in\Gamma_{\mathrm{T}}$ by~$g'_{0}(s):=\lim_{t \to 0^{+}}\frac{g_{t}(s)-g_{0}(s)}{t}$.

We emphasize that Signorini's unilateral conditions are obtained on the boundaries~$\Gamma^{u_{0},g_{0}}_{\mathrm{T}_{\mathrm{S-}}}$ and~$\Gamma^{u_{0},g_{0}}_{\mathrm{T}_{\mathrm{S+}}}$ (see, e.g.,~\cite[Section~1]{LIONS2} for a similar scalar version of Signorini's unilateral conditions), although the sensitivity analysis focused on the Tresca's friction law perturbation. Hence, our work reveals an unexpected link between these two classical boundary conditions in contact mechanics. Roughly speaking, our main result claims that Signorini’s solution can be considered as first-order approximation of perturbed Tresca’s solutions. Precisely, for small values of $t>0$, the Tresca’s solution $u_{t}$ can be approximated by $u_{0}+tu_{0}'$ in~$\HH^{1}$-norm. Some numerical simulations are provided at the end of the paper in order to illustrate this result.

\paragraph{Organization of the paper.}
The paper is organized as follows. Section~\ref{rappel} is dedicated to some basic notions from convex analysis and functional analysis used throughout the paper. Section~\ref{Mainresult} is the core of the present work. In Section~\ref{section5} we introduce three boundary value problems involved in the sensitivity analysis of the Tresca friction problem performed in Section~\ref{section4} which is concluded with our main result (see Theorem~\ref{caractu0derivDNT}). In Section~\ref{simunum}, numerical simulations are provided in order to illustrate our main result. Finally, in Appendix~\ref{casparticulier}, some sufficient conditions that guarantee the twice epi-differentiability of the parameterized Tresca friction functional, which is a crucial assumption in our main theorem, are provided.

\section{Notions from Convex Analysis and Functional Framework }\label{rappel}
In this section we start with some notions from convex analysis in Section~\ref{rappelconvex} and we conclude with some basics of functional analysis in Section~\ref{rappelfunct}.

\subsection{Notions from Convex Analysis}\label{rappelconvex}
For notions and results presented in this section, we refer to standard references such as~\cite{BREZ2,MINTY,ROCK2} and~\cite[Chapter~12]{ROCK}. In what follows $(\mathrm{V}, \dual{\cdot}{\cdot}_{\mathrm{V}})$ stands for a general real Hilbert space. 
\begin{myDefn}[Domain and epigraph]\label{epidom}
Let  $\phi \,: \, \mathrm{V}\rightarrow \mathbb{R}\cup\left\{\pm \infty \right\}$.
The \emph{domain} and the \emph{epigraph} of~$\phi$ are respectively defined by
$$
\mathrm{dom}\left(\phi\right):=\left\{x\in \mathrm{V} \mid \phi(x)<+\infty \right\} \quad \text{and} \quad
\mathrm{epi}\left(\phi\right):=\left\{(x,t)\in \mathrm{V}\times\mathbb{R}\mid \phi(x)\leq t\right\}.
$$
\end{myDefn}
Recall that $\phi \,: \, \mathrm{V}\rightarrow \mathbb{R}\cup\left\{\pm \infty \right\}$ is said to be \textit{proper} if $\mathrm{dom}(\phi)\neq \emptyset$ and $\phi(x)>-\infty$ 
for all~$x\in\mathrm{V}$. Moreover, $\phi$ is a convex (resp.\ lower semi-continuous) function on $\mathrm{V}$ if and only if $\mathrm{epi}(\phi)$ is a convex (resp.\ closed) subset of~$\mathrm{V}\times\R$.
\begin{myDefn}[Convex subdifferential operator]\label{sousdiff}
 Let $\phi  :  \mathrm{V} \rightarrow \R\cup\left\{+\infty\right\}$ be a proper function. We denote by $\partial{\phi}  :  \mathrm{V} \rightrightarrows \mathrm{V}$ the \emph{convex subdifferential operator} of $\phi$, defined by 
 $$
 \partial{\phi}(x):=\left\{y\in\mathrm{V} \mid \forall z\in\mathrm{V}\text{, } \dual{y}{z-x}_{\mathrm{V}}\leq \phi(z)-\phi(x)\right\},
 $$
for all $x\in\mathrm{V}$.
\end{myDefn} 
\begin{myDefn}[Proximal operator]\label{proxi}
 Let $\phi  :  \mathrm{V} \rightarrow \R\cup\left\{+\infty\right\}$ be a proper lower semi-continuous convex function. The \emph{proximal operator} associated with $\phi$ is the map~$\mathrm{prox}_{\phi}  :  \mathrm{V} \rightarrow \mathrm{V}$ defined by
 $$
      \mathrm{prox}_{\phi}(x):=\underset{y\in \mathrm{V}}{\argmin}\left[ \phi(y)+\frac{1}{2}\left \| y-x \right \|^{2}_{\mathrm{V}}\right]=(\mathrm{I}+\partial \phi)^{-1}(x),
 $$
for all $x\in\mathrm{V}$, where $\mathrm{I}  :  \mathrm{V}\rightarrow \mathrm{V}$ stands for the identity operator.
\end{myDefn}
The proximal operator have been introduced by J.-J. Moreau in 1965 (see~\cite{MOR}) and can be seen as a generalization of the classical projection operators onto nonempty closed convex subsets. It is well-known that, if $\phi  :  \mathrm{V} \rightarrow \R\cup\left\{+\infty\right\}$ is a proper lower semi-continuous convex function, then~$\partial{\phi}$ is a maximal monotone operator (see, e.g.,~\cite{ROCK2}), and thus the proximal operator~$\mathrm{prox}_{\phi}$ is well-defined and a single-valued map (see, e.g.,~\cite[Chapter II]{BREZ2}).

In what follows, some definitions related to the notion of twice epi-differentiability are recalled (for more details, see~\cite[Chapter 7, section B p.240]{ROCK} for the finite-dimensional case and~\cite{DO} for the infinite-dimensional one). The strong (resp.\ weak) convergence of a sequence in~$\mathcal{H}$ will be denoted by~$\rightarrow$ (resp.\ $\rightharpoonup$) and note that all limits with respect to~$t$ will be considered for~$t \to 0^+$.
\begin{myDefn}[Mosco-convergence]\label{limitemuch}
The \emph{outer}, \emph{weak-outer}, \emph{inner} and \emph{weak-inner limits} of a parameterized family~$(A_{t})_{t>0}$ of subsets of $\mathrm{V}$ are respectively defined by
\begin{eqnarray*}
      \mathrm{lim}\sup A_{t}&:=&\left\{ x\in \mathrm{V} \mid \exists (t_{n})_{n\in\mathbb{N}}\rightarrow 0^{+}, \exists \left(x_{n}\right)_{n\in\mathbb{N}}\rightarrow x, \forall n\in\mathbb{N}, x_{n}\in A_{t_{n}}\right\},\\
     \mathrm{w}\text{-}\mathrm{lim}\sup A_{t}&:=&\left\{ x\in \mathrm{V} \mid \exists (t_{n})_{n\in\mathbb{N}}\rightarrow 0^{+}, \exists \left(x_{n}\right)_{n\in\mathbb{N}}\rightharpoonup x, \forall n\in\mathbb{N}, x_{n}\in A_{t_{n}}\right\},\\
     \mathrm{lim}\inf A_{t}&:=&\left\{ x\in \mathrm{V} \mid \forall (t_{n})_{n\in\mathbb{N}}\rightarrow 0^{+}, \exists \left(x_{n}\right)_{n\in\mathbb{N}}\rightarrow x, \exists N\in\mathbb{N}, \forall n\geq N, x_{n}\in A_{t_{n}}\right\},\\
     \mathrm{w}\text{-}\mathrm{lim}\inf A_{t}&:=&\left\{ x\in \mathrm{V} \mid \forall (t_{n})_{n\in\mathbb{N}}\rightarrow 0^{+}, \exists \left(x_{n}\right)_{n\in\mathbb{N}}\rightharpoonup x, \exists N\in\mathbb{N}, \forall n\geq N, x_{n}\in A_{t_{n}}\right\}.
\end{eqnarray*}
The family~$(A_{t})_{t>0}$ is said to be \emph{Mosco-convergent} if~$
\mathrm{w}\text{-}\mathrm{lim}\sup A_{t}\subset\mathrm{lim}\inf A_{t}
$. In that case, all the previous limits are equal and we write
$$
     \mathrm{M}\text{-}\mathrm{lim} A_{t}:=\mathrm{lim}\inf A_{t}=\mathrm{lim}\sup A_{t}=\mathrm{w}\text{-}\mathrm{lim}\inf A_{t}=\mathrm{w}\text{-}\mathrm{lim}\sup A_{t}.
$$
\end{myDefn}
\begin{myDefn}[Mosco epi-convergence]
  Let $(\phi_{t})_{t>0}$ be a parameterized family of functions~$\phi_{t}  : \mathrm{V}\rightarrow \mathbb{R}\cup\left\{\pm \infty \right\}$ for all $t>0$.
 We say that $(\phi_{t})_{t>0}$ is \emph{Mosco epi-convergent} if~$(\mathrm{epi}(\phi_{t}))_{t>0}$ is Mosco-convergent in~$\mathrm{V} \times \R$. Then we denote by $\mathrm{ME}\text{-}\mathrm{lim}~ \phi_{t}  :  \mathrm{V}\rightarrow \mathbb{R}\cup\left\{\pm \infty \right\}$ the function characterized by its epigraph~$\mathrm{epi}\left(\mathrm{ME}\text{-}\mathrm{lim}~\phi_{t}\right):=\mathrm{M}\text{-}\mathrm{lim}$ $\displaystyle \mathrm{epi}\left(\phi_{t}\right)$ and we say that $(\phi_{t})_{t>0}$ Mosco epi-converges to~$\mathrm{ME}\text{-}\mathrm{lim}~\phi_{t}$.
 \end{myDefn}
 
The proof of the next proposition can be found in~\cite[Proposition 3.19 p.297]{ATTOUCH}.
\begin{myProp}[Characterization of Mosco epi-convergence]\label{caractMosco}
Let $(\phi_{t})_{t>0}$ be a parameterized family of functions~$\phi_{t}  :  \mathrm{V}\rightarrow \mathbb{R}\cup\left\{\pm \infty \right\}$ for all $t>0$ and let~$\phi  :  \mathrm{V}\rightarrow \mathbb{R}\cup\left\{\pm \infty \right\}$. Then $(\phi_{t})_{t>0}$ Mosco epi-converges to $\phi$ if and only if, for all $x\in \mathrm{V}$, the two conditions
\begin{enumerate}
    \item there exists $(x_{t})_{t>0}\rightarrow x$ such that $\mathrm{lim}\sup \phi_{t}(x_{t})\leq \phi(x)$;
    \item for all $(x_{t})_{t>0}\rightharpoonup  x$, $\mathrm{lim}\inf \phi_{t}(x_{t})\geq \phi(x)$;
\end{enumerate}
are satisfied.
\end{myProp}

Now let us recall the notion of twice epi-differentiability introduced by R.T.~Rockafellar in~1985 (see~\cite{Rockafellar}) that generalizes the classical notion of second-order derivative to nonsmooth convex functions.
 \begin{myDefn}[Twice epi-differentiability]\label{epidiff}
   A proper lower semi-continuous convex function~$\phi  : \mathrm{V}\rightarrow \mathbb{R}\cup\left\{+\infty \right\}$ is said to be \emph{twice epi-differentiable} at $x\in\mathrm{dom}(\phi)$ for $y\in\partial\phi(x)$ if the family of second-order difference quotient functions $(\delta_{t}^{2}\phi(x|y))_{t>0}$ defined by
  $$
  \fonction{ \delta_{t}^{2}\phi(x|y) }{\mathrm{V}}{\mathbb{R}\cup\left\{+\infty\right\}}{z}{\displaystyle\frac{\phi(x+t z)-\phi(x)-t\dual{ y}{z}_{\mathrm{V}}}{t^{2}},}
$$
for all $t>0$, is Mosco epi-convergent. In that case we denote by
$$
\mathrm{d}_{e}^{2}\phi(x|y):=\mathrm{ME}\text{-}\mathrm{lim}~\delta_{t}^{2}\phi(x|y),
$$
which is called the second-order epi-derivative of $\phi$ at $x$ for $y$.
\end{myDefn}
\begin{myExa}\label{epidiffabs}
The classical absolute value map $\left|\cdot\right|  :  \mathbb{R}\rightarrow\mathbb{R}$, which is a proper lower semi-continuous convex function on $\mathbb{R}$, is twice epi-differentiable at any $x\in\R$ for any~$y\in\partial{\left|\cdot\right|}(x)$, and its second-order epi-derivative is given by~$ \mathrm{d}_{e}^{2}|\mathord{\cdot} |(x|y)=\mathrm{I}_{\mathrm{K}_{x,y}}$, where $\mathrm{K}_{x,y}$ is the nonempty closed convex subset of $\R$ defined by
$$
\mathrm{K}_{x,y}:=
\left\{
\begin{array}{lcll}
\R	&   & \text{ if } x\neq0 , \\
\mathbb{R}^{+}	&   & \text{ if } x=0 \text{ and } y=1 , \\
\mathbb{R}^{-}	&   & \text{ if } x=0 \text{ and } y=-1 , \\
\left\{0\right\}	&   & \text{ if } x=0 \text{ and } y\in (-1,1),
\end{array}
\right.
$$
and where~$\mathrm{I}_{\mathrm{K}_{x,y}}$ stands for the indicator function of~$\mathrm{K}_{x,y}$, defined by $\mathrm{I}_{\mathrm{K}_{x,y}}(z):=0$ if $z\in\mathrm{K}_{x,y}$, and~$\mathrm{I}_{\mathrm{K}_{x,y}}(z):=+\infty$ otherwise (see~\cite[Example 2.6. p.7]{4ABC}).
\end{myExa}

In the above classical definition of twice epi-differentiability, the function $\phi$ does not depend on the parameter~$t$. However, in this paper, the parameterized Tresca friction functional does (see Introduction). Therefore, in this paper, we will use an extended version of twice epi-differentiability which has been recently introduced in~\cite{8AB}. To this aim, when considering a function $\Phi  :  \mathbb{R}^{+}\times \mathrm{V}\rightarrow \mathbb{R}\cup\left\{+\infty\right\}$ such that, for all $t\geq0$, $\Phi(t,\cdot)$ is a proper function on $\mathrm{V}$, we will make use of the two following notations: $\partial \Phi(0,\mathord{\cdot} )(x)$ stands for the convex subdifferential operator at~$x\in\mathrm{V}$ of the map~$w\in\mathrm{V} \mapsto \Phi(0,w)\in \R\cup\left\{+\infty\right\}$, and~$\Phi^{-1}(\mathord{\cdot} , \mathbb{R}):=\left\{ x\in\mathrm{V}\mid \forall t\geq0, \; \Phi(t,x)\in\R \right\}$.

 \begin{myDefn}[Twice epi-differentiability depending on a parameter]\label{epidiffpara}
Let~$\Phi  :  \mathbb{R}^{+}\times \mathrm{V}\rightarrow \mathbb{R} \cup\left\{+\infty\right\}$ be a function such that, for all $t\geq0$, $\Phi(t,\cdot)$ is a proper lower semi-continuous convex function on $\mathrm{V}$. The function $\Phi$ is said to be \emph{twice epi-differentiable} at $x\in \Phi^{-1}(\mathord{\cdot} , \mathbb{R})$ for $y\in\partial \Phi(0,\mathord{\cdot} )(x)$ if the family of second-order difference quotient functions $(\Delta_{t}^{2}\Phi(x|y))_{t>0}$ defined by
$$
  \fonction{\Delta_{t}^{2}\Phi(x|y) }{\mathrm{V}}{\mathbb{R}\cup\left\{+\infty\right\}}{z}{\displaystyle\frac{\Phi(t,x+t z)-\Phi(t,x)-t\dual{ y}{z}_{\mathrm{V}}}{t^{2}},}
$$
for all $t>0$, is Mosco epi-convergent. In that case, we denote by
$$
\mathrm{D}_{e}^{2}\Phi(x|y):=\mathrm{ME}\text{-}\mathrm{lim}~\Delta_{t}^{2}\Phi(x|y) ,
$$
which is called the second-order epi-derivative of $\Phi$ at $x$ for $y$.
\end{myDefn}
Note that, if the function $\Phi$ is $t$-independent in Definition~\ref{epidiffpara}, then we recover Definition~\ref{epidiff}. Finally the following theorem is the key point in order to derive our main result in this paper. It is a particular case of a more general theorem that can be found in~\cite[Theorem 4.15 p.1714]{8AB}.
\begin{myTheorem}\label{TheoABC2018}
Let~$\Phi  :  \mathbb{R}^{+}\times \mathrm{V}\rightarrow \mathbb{R} \cup \left\{+\infty\right\}$ be a function such that, for all $t\geq0$, $\Phi(t,\cdot)$ is a proper lower semi-continuous convex function on $\mathrm{V}$. Let $F  :  \mathbb{R}^{+}\rightarrow \mathrm{V}$ and let~$u  :  \mathbb{R}^{+}\rightarrow \mathrm{V}$ be defined by
$$
    u(t):=\mathrm{prox}_{\Phi(t,\mathord{\cdot} )}(F(t)),
$$
for all~$t\geq 0$. If the conditions 
\begin{enumerate}
    \item $F$ is differentiable at $t=0$;
    \item $\Phi$ is twice epi-differentiable at $u(0)$ for $F(0)-u(0)\in\partial \Phi(0,\mathord{\cdot} )(u(0))$;
    \item $\mathrm{D}_{e}^{2}\Phi(u(0)|F(0)-u(0))$ is a proper lower semi-continuous convex function on $\mathrm{V}$;
\end{enumerate}
are satisfied, then $u$ is differentiable at $t=0$ with
$$
u'(0)=\mathrm{prox}_{\mathrm{D}_{e}^{2}\Phi(u(0)|F(0)-u(0))}(F'(0)).
$$
\end{myTheorem}
\subsection{Functional Framework}\label{rappelfunct}
Let $d\geq1$ and $\Omega$ be a nonempty bounded connected open subset of $\R^{d}$ with a Lipschitz continuous boundary $\Gamma:=\partial{\Omega}$. We denote by $\LL^{2}(\Omega)$, $\LL^{2}(\Gamma)$, $\LL^{1}(\Gamma)$, $\HH^{1}(\Omega)$, $\HH^{1/2}(\Gamma)$, $\HH^{-1/2}(\Gamma)$ the usual Lebesgue and Sobolev spaces endowed with their standard norms. 
Moreover~the~no\-tation~$\MD(\Omega)$ stands for the set of functions $\varphi  :  \Omega \rightarrow \R$ that are infinitely differentiable with compact support in $\Omega$, and $\MD'(\Omega)$ for the set of distributions on $\Omega$.
In what follows we consider a decomposition~$\Gamma=:\Gamma_{1}\cup\Gamma_{2}$ where $\Gamma_{1}$ and $\Gamma_{2}$ are two measurable disjoint subsets of $\Gamma$. Let us recall some classical embeddings useful in this paper, that can be found for instance in~\cite{BREZ} and \cite[Chapter~7, Section~2 p.395]{DAUTLIONS}.


\begin{myProp}\label{injections}
The continuous and dense embeddings~$\HH^{1}(\Omega){\hookrightarrow} \HH^{1/2}(\Gamma){\hookrightarrow} \LL^{2}(\Gamma){\hookrightarrow} \HH^{-1/2}(\Gamma)$, $\LL^{2}(\Gamma){\hookrightarrow} \LL^{1}(\Gamma)$, $\HH^{1}(\Omega){\hookrightarrow} \mathrm{L}^{2}(\Omega)$, and $\HH^{1/2}_{00}(\Gamma_{1}){\hookrightarrow} \LL^{2}(\Gamma_{1}){\hookrightarrow} \HH^{-1/2}_{00}(\Gamma_{1})$ are satisfied, where $\HH^{1/2}_{00}(\Gamma_{1})$ is the vector subspace of $\HH^{1/2}(\Gamma)$ defined by
$$
\HH^{1/2}_{00}(\Gamma_{1}):=\left\{v\in\LL^{2}(\Gamma_{1}) \mid \exists w\in\HH^{1}(\Omega), \; w=v \text{ on }\Gamma_{1} \text{ and } w=0 \text{ on }\Gamma_{2} \right\},
$$
and $\HH^{-1/2}_{00}(\Gamma_{1})$ stands for its dual space.
Furthermore the dense and compact embedding~$\HH^{1}(\Omega)\hookdoubleheadrightarrow~\mathrm{L}^{2}(\Gamma)$ holds true.
\end{myProp}
  
The next proposition is a particular case of a more general statement that can be found in~\cite[Section 2.9 p.56]{TUC}.
\begin{myProp}\label{Ident}
Let $w\in\HH^{-1/2}_{00}(\Gamma_{1})$. If there exists $C\geq 0$ such that
$$
\dual{w}{v}_{\HH^{-1/2}_{00}(\Gamma_{1})\times \HH^{1/2}_{00}(\Gamma_{1})}\leq C\left \| v \right \|_{\LL^{2}(\Gamma_{1})},
$$
for all $v\in \HH^{1/2}_{00}(\Gamma_{1})$, then $w$ can be identified to an element $h\in \LL^{2}(\Gamma_{1})$ with $$\dual{w}{v}_{\HH^{-1/2}_{00}(\Gamma_{1})\times \HH^{1/2}_{00}(\Gamma_{1})}=\dual{h}{v}_{\LL^{2}(\Gamma_{1})},
$$
for all $v\in\HH^{1/2}_{00}(\Gamma_{1})$.
\end{myProp}

The next proposition, known as Green formula, can be found in~\cite[Corollary 2.6 p.28]{GIR}.
\begin{myProp}[Green formula] \label{Green}
Let $w\in \HH^{1}(\Omega)$. If $\Delta w\in \LL^{2}(\Omega)$, then~$\nabla w$ admits a normal trace $\partial_{\nn}w\in \HH^{-1/2}(\Gamma)$ such that
$$
\int_{\Omega}\varphi\Delta w+\int_{\Omega}\nabla w\cdot\nabla\varphi=\dual{\partial_{\nn}w}{\varphi}_{\HH^{-1/2}(\Gamma)\times \HH^{1/2}(\Gamma)},
$$
for all $\varphi\in\HH^{1}(\Omega)$.
\end{myProp}

\section{Main Result}\label{Mainresult}
In this section we establish the main result of the paper (Theorem~\ref{caractu0derivDNT}). To this aim, we introduce in Section~\ref{section5} three problems with different boundary conditions: Dirichlet and Neumann conditions (see Problem~\eqref{PbNeumannDirichlet}), Signorini unilateral conditions (see Problem~\eqref{PbDirichletNeumannSigno}) and a condition with the Tresca's friction law (see Problem~\eqref{PbNeumannDirichletTresca}). For each problem, strong and weak solutions are defined, then the equivalence between both solutions is investigated and the existence/uniqueness of the weak solution is proved. All these preliminary results are used in the next Section~\ref{section4} in order to perform the sensitivity analysis of the Tresca friction problem. 

In what follows, let $d\geq1$ and $\Omega$ be a nonempty bounded connected open subset of $\R^{d}$ with a Lipschitz continuous boundary $\Gamma:=\partial{\Omega}$.
We consider the decomposition 
$$
\Gamma=:\Gamma_{\mathrm{D}}\cup\Gamma_{\mathrm{N}}\cup\Gamma_{\mathrm{T}}, 
$$
where $\Gamma_{\mathrm{D}}$, $\Gamma_{\mathrm{N}}$, $\Gamma_{\mathrm{T}}$ are three measurable pairwise disjoint subsets of $\Gamma$, such that $\Gamma_{\mathrm{D}}$ and $\Gamma_{\mathrm{T}}$ have a positive measure. 
Let us denote by $B_{\Gamma}(s,\varepsilon)$ the open ball of $\Gamma$ centered at some $s\in\Gamma$ with some radius $\varepsilon>0$, and, for some subset $\mathrm{A}$ of $\Gamma$, by $\mathrm{int}_{\Gamma}(\mathrm{A})$ the interior of $\mathrm{A}$ in $\Gamma$. Moreover, for any function $v\in\HH^{1}(\Omega)$ such that $\Delta v\in\LL^{2}(\Omega)$, we denote by $\partial_{\nn}v$ the function in $\HH^{-1/2}(\Gamma)$ given by the Green Formula (see Proposition~\ref{Green}).

In this paper we work with the closed vector subspace $\HH^{1}_{\mathrm{D}}(\Omega)$ of $\HH^{1}(\Omega)$ defined by
$$
\HH^{1}_{\mathrm{D}}(\Omega):=\left\{ \varphi\in\HH^{1}(\Omega) \mid \varphi=0 \text{ almost everywhere on } \Gamma_{\mathrm{D}}\right\}.
$$
Since $\Gamma_{\mathrm{D}}$ has a positive measure and thanks to Poincaré's inequality, note that $(\HH^{1}_{\mathrm{D}}(\Omega), \dual{\cdot}{\cdot}_{\HH^{1}_{\mathrm{D}}(\Omega)})$ is a real Hilbert space endowed with the scalar product
$$
\fonction{\dual{\cdot}{\cdot}_{\HH^{1}_{\mathrm{D}}(\Omega)}}{\HH^{1}_{\mathrm{D}}(\Omega)\times\HH^{1}_{\mathrm{D}}(\Omega)}{\R}{(u,v)}{\displaystyle \int_{\Omega}\nabla u\cdot\nabla v.}
$$
The corresponding norm, which is equivalent to the norm~$\left\|\cdot\right\|_{\HH^{1}(\Omega)}$ on~$\HH^{1}_{\mathrm{D}}(\Omega)$, will be denoted by~$\left\|\cdot\right\|_{\HH^{1}_{\mathrm{D}}(\Omega)}$.
\subsection{Some Required Boundary Value Problems
}\label{section5}
For the needs of this preliminary section, let us fix some functions $f\in \mathrm{L}^{2}(\Omega)$, $k\in \mathrm{L}^{2}(\Gamma_{\mathrm{N}})$, $h\in\LL^{2}(\Gamma_{\mathrm{T}})$ and $g\in \mathrm{L}^{2}(\Gamma_{\mathrm{T}})$ such that $g>0$ almost everywhere on $\Gamma_{\mathrm{T}}$. Here the proofs are very close to the ones presented in the recent paper~\cite{4ABC}, and thus they are left to the reader. The main differences being the use of the Hilbert space $\HH^{1}_{\mathrm{D}}(\Omega)$ instead of $\HH^{1}(\Omega)$, and the vector subspace~$\HH^{1/2}_{00}(\Gamma_{\mathrm{N}}\cup\Gamma_{\mathrm{T}})$ of $\HH^{1/2}(\Gamma)$ instead of $\HH^{1/2}(\Gamma)$.

\subsubsection{A Problem with Dirichlet-Neumann Conditions}
Let us consider the Dirichlet-Neumann problem given by
\begin{equation}\tag{DN} \label{PbNeumannDirichlet}
\arraycolsep=2pt
\left\{
\begin{array}{rcll}
-\Delta & = & f   & \text{ in } \Omega , \\
F & = & 0  & \text{ on } \Gamma_{\mathrm{D}} ,\\
\partial_{\nn} F & = & k  & \text{ on } \Gamma_{\mathrm{N}} ,\\
\partial_{\nn} F & = & h  & \text{ on } \Gamma_{\mathrm{T}}.
\end{array}
\right.
\end{equation}
Considering two different Neumann conditions in Problem~\eqref{PbNeumannDirichlet} may seem artificial to the reader. Nevertheless this framework naturally appears in our main result (see Problem~\eqref{caractu0DNT} in Theorem~\ref{caractu0derivDNT}). This is the reason why we consider it in this preliminary section. 

\begin{myDefn}[Strong solution to the Dirichlet-Neumann problem]
A strong solution to the Dirichlet-Neumann problem~\eqref{PbNeumannDirichlet} is a function $F\in\HH^{1}(\Omega)$ such that $-\Delta F=f$ in $\MD'(\Omega)$,~$F=0$ almost everywhere on $\Gamma_{\mathrm{D}}$, $\partial_{\nn}F\in\LL^{2}(\Gamma_{\mathrm{N}}\cup\Gamma_{\mathrm{T}})$ with $\partial_{\nn}F=k$ almost everywhere on $\Gamma_{\mathrm{N}}$ and $\partial_{\nn}F=h$ almost everywhere on $\Gamma_{\mathrm{T}}$.
\end{myDefn}
\begin{myDefn}[Weak solution to the Dirichlet-Neumann problem]\label{SolufaibleDN}
A weak solution to the Dirichlet-Neumann problem~\eqref{PbNeumannDirichlet} is a function $F\in\HH^{1}_{\mathrm{D}}(\Omega)$ such that
\begin{equation}\label{FaibleDirichNeumann}
    \displaystyle\int_{\Omega}\nabla F\cdot\nabla v=\int_{\Omega}fv+\int_{\Gamma_{\mathrm{N}}}kv+\int_{\Gamma_{\mathrm{T}}}hv, \qquad \forall v\in\HH^{1}_{\mathrm{D}}(\Omega).
\end{equation}
\end{myDefn}
\begin{myProp}\label{DNequiart}
A function $F\in\HH^{1}(\Omega)$ is a strong solution to the Dirichlet-Neumann problem~\eqref{PbNeumannDirichlet} if and only if $F$ is a weak solution to the Dirichlet-Neumann problem~\eqref{PbNeumannDirichlet}.
\end{myProp}

Using the Riesz representation theorem, we obtain the following existence/uniqueness result.
\begin{myProp}\label{existenceunicitéDN}
The Dirichlet-Neumann problem~\eqref{PbNeumannDirichlet} admits a unique solution $F\in\HH^{1}_{\mathrm{D}}(\Omega)$. 
Moreover there exists a constant $C \geq 0$ (depending only on $\Omega$) such that
$$
\left \| F \right \|_{\HH^{1}_{\mathrm{D}}(\Omega)}\leq C\left(\left \| f  \right \|_{\LL^{2}(\Omega)} + \left \| k \right \|_{\LL^{2}(\Gamma_{\mathrm{N}})}+\left \| h \right \|_{\LL^{2}(\Gamma_{\mathrm{T}})}\right).
$$
\end{myProp}

\subsubsection{A Signorini Problem}\label{SectionSignorinicasscalairesansu}
Here we assume that $\Gamma_{\mathrm{T}}$ can be decomposed as
$$
\Gamma_{\mathrm{T}}=:\Gamma_{\mathrm{T}_{\mathrm{S}_{\mathrm{N}}}}\cup\Gamma_{\mathrm{T}_{\mathrm{S}_{\mathrm{D}}}}\cup\Gamma_{\mathrm{T}_{\mathrm{S-}}}\cup\Gamma_{\mathrm{T}_{\mathrm{S+}}},
$$
where $\Gamma_{\mathrm{T}_{\mathrm{S}_{\mathrm{N}}}}$, $\Gamma_{\mathrm{T}_{\mathrm{S}_{\mathrm{D}}}}$, $\Gamma_{\mathrm{T}_{\mathrm{S-}}}$ and $\Gamma_{\mathrm{T}_{\mathrm{S+}}}$ are four measurable pairwise disjoint subsets of $\Gamma$, and we consider the Signorini problem given by
\begin{equation}\tag{SP} \label{PbDirichletNeumannSigno}
\arraycolsep=2pt
\left\{
\begin{array}{rcll}
-\Delta u & = &  f   & \text{ in } \Omega , \\
u & = & 0  & \text{ on } \Gamma_{\mathrm{D}}\cup\Gamma_{\mathrm{T}_{\mathrm{S}_{\mathrm{D}}}}, \\
\partial_{\nn} u & = & k  & \text{ on } \Gamma_{\mathrm{N}}, \\
\partial_{\nn} u & = & h  & \text{ on } \Gamma_{{\mathrm{T}_{\mathrm{S}_{\mathrm{N}}}}}, \\
 u\leq0\text{, } \partial_{\nn}u\leq h \text{ and } u\left(\partial_{\nn}u-h\right)	& = & 0  & \text{ on } \Gamma_{\mathrm{T}_{\mathrm{S-}}}, \\
 u\geq0\text{, } \partial_{\nn}u\geq h \text{ and } u\left(\partial_{\nn}u-h\right)	& = & 0  & \text{ on } \Gamma_{\mathrm{T}_{\mathrm{S+}}}.
\end{array}
\right.
\end{equation}
\begin{myDefn}[Strong solution to the Signorini problem]
  A strong solution to the Signorini problem~\eqref{PbDirichletNeumannSigno} is a function $u\in \HH^{1}(\Omega)$ such that $-\Delta u=f$ in $\mathcal{D}
'(\Omega)$, $u=0$ on $\Gamma_{\mathrm{D}}\cup\Gamma_{\mathrm{T}_{\mathrm{S}_{\mathrm{D}}}}$, and also~$\partial_{\nn}u\in \mathrm{L}^{2}(\Gamma_{\mathrm{N}}\cup\Gamma_{\mathrm{T}})$ with $\partial_{\nn}u=k$ almost everywhere on $\Gamma_{\mathrm{N}}$, $\partial_{\nn}u=h$ almost everywhere on~$\Gamma_{\mathrm{T}_{\mathrm{S}_{\mathrm{N}}}}$, $u\leq0$, $\partial_{\nn}u\leq h$ and $u(\partial_{\nn}u-h)=0$ almost everywhere on $\Gamma_{\mathrm{T}_{\mathrm{S-}}}$, $u\geq0$, $\partial_{\nn}u\geq h$ and~$u(\partial_{\nn}u-h)=0$ almost everywhere on $\Gamma_{\mathrm{T}_{\mathrm{S+}}}$.
\end{myDefn}
\begin{myDefn}[Weak solution to the Signorini problem]\label{WeaksolutionSigno}
A weak solution to the Signorini problem~\eqref{PbDirichletNeumannSigno} is a function $u\in\mathcal{K}^{1}(\Omega)$ such that 
\begin{equation}\label{FormuFaibleDNS}
  \displaystyle\int_{\Omega}\nabla u\cdot\nabla(v-u)\geq\int_{\Omega}f(v-u)+\int_{\Gamma_{\mathrm{N}}}k(v-u)+\int_{\Gamma_{\mathrm{T}}}h(v-u),\qquad \forall v\in\mathcal{K}^{1}(\Omega),
\end{equation}
where $\mathcal{K}^{1}(\Omega)$ is the nonempty closed convex subset of $\HH^{1}_{\mathrm{D}}(\Omega)$ given by
$$
\mathcal{K}^{1}(\Omega) := \left\{v\in\HH^{1}_{\mathrm{D}}(\Omega) \mid v\leq 0 \text{ on } \Gamma_{\mathrm{T}_{\mathrm{S-}}}\text{, } v=0 \text{ on }\Gamma_{\mathrm{T}_{\mathrm{S}_{\mathrm{D}}}} \text{ and } v\geq0 \text{ on } \Gamma_{\mathrm{T}_{\mathrm{S+}}} \right \}.
$$
\end{myDefn}

One can easily prove that a strong solution is a weak solution but, to the best of our knowledge, without additional assumptions, one cannot prove the converse. To get the equivalence, we need to assume, in particular, that the decomposition of $\Gamma$ is \textit{consistent} in the following sense.
\begin{myDefn}[Consistent decomposition]\label{regulieresens2}
 The decomposition   $\Gamma=\Gamma_{\mathrm{D}}\cup\Gamma_{\mathrm{N}}\cup\Gamma_{\mathrm{T}_{\mathrm{S}_{\mathrm{N}}}}\cup\Gamma_{\mathrm{T}_{\mathrm{S}_{\mathrm{D}}}}\cup\Gamma_{\mathrm{T}_{\mathrm{S-}}}\cup\Gamma_{\mathrm{T}_{\mathrm{S+}}}$ is said to be \emph{consistent} if
 \begin{enumerate}
     \item for almost all $s\in\Gamma_{\mathrm{T}_{\mathrm{S-}}}\cup\Gamma_{\mathrm{T}_{\mathrm{S+}}}$,  $s\in \mathrm{int}_{\Gamma}(\Gamma_{\mathrm{T}_{\mathrm{S-}}})$ or $s\in\mathrm{int}_{\Gamma}(\Gamma_{\mathrm{T}_{\mathrm{S+}}})$;
     \item the set $\mathcal{K}^{1/2}(\Gamma)$ given by
    $$         
    \mathcal{K}^{1/2}(\Gamma):=\left \{ v\in \HH^{1/2}(\Gamma) \mid v\leq 0 \text{ on } \Gamma_{\mathrm{T}_{\mathrm{S-}}}\text{, } v=0 \text{ on }\Gamma_{\mathrm{D}}\cup\Gamma_{\mathrm{T}_{\mathrm{S}_{\mathrm{D}}}} \text{ and } v\geq0 \text{ on } \Gamma_{\mathrm{T}_{\mathrm{S+}}} \right \},
    $$
(which is a nonempty closed convex subset of $\HH^{1/2}(\Gamma)$) is dense in  the nonempty closed convex subset $\mathcal{K}^{0}(\Gamma)$ of $\mathrm{L}^{2}(\Gamma)$ given by
     $$
         \mathcal{K}^{0}(\Gamma):=\left \{ v\in \mathrm{L}^{2}(\Gamma) \mid v\leq 0 \text{ on } \Gamma_{\mathrm{T}_{\mathrm{S-}}}\text{, } v=0 \text{ on }\Gamma_{\mathrm{D}}\cup\Gamma_{\mathrm{T}_{\mathrm{S}_{\mathrm{D}}}} \text{ and } v\geq0 \text{ on } \Gamma_{\mathrm{T}_{\mathrm{S+}}} \right \}.
     $$
\end{enumerate}
\end{myDefn}
\begin{myProp}\label{EquiSignoCassansu}
Let $u\in \HH^{1}(\Omega)$.
\begin{enumerate}
    \item If $u$ is a strong solution to the Signorini problem~\eqref{PbDirichletNeumannSigno}, then $u$ is a weak solution to the Signorini problem~\eqref{PbDirichletNeumannSigno}.
    \item If $u$ is a weak solution to the Signorini problem~\eqref{PbDirichletNeumannSigno} such that $\partial_{\nn}u\in \mathrm{L}^{2}(\Gamma_{\mathrm{N}}\cup\Gamma_{\mathrm{T}})$ and the decomposition $\Gamma=\Gamma_{\mathrm{D}}\cup\Gamma_{\mathrm{N}}\cup\Gamma_{\mathrm{T}_{\mathrm{S}_{\mathrm{N}}}}\cup\Gamma_{\mathrm{T}_{\mathrm{S}_{\mathrm{D}}}}\cup\Gamma_{\mathrm{T}_{\mathrm{S-}}}\cup\Gamma_{\mathrm{T}_{\mathrm{S+}}}$ is consistent, then $u$ is a strong solution to the Signorini problem~\eqref{PbDirichletNeumannSigno}.
\end{enumerate}
\end{myProp}

Using the classical characterization of the projection operator, we obtain the following existence/uniqueness result.
\begin{myProp}\label{existenceunicitepbDNS}
The Signorini problem~\eqref{PbDirichletNeumannSigno} admits a unique weak solution $u\in\HH^{1}_{\mathrm{D}}(\Omega)$ which is given by
$$
    \displaystyle u=\mathrm{proj}_{\mathcal{K}^{1}(\Omega)}(F),
$$
where $F\in\HH^{1}_{\mathrm{D}}(\Omega)$ is the solution to the Dirichlet-Neumann problem~\eqref{PbNeumannDirichlet}, and $\mathrm{proj}_{\mathcal{K}^{1}(\Omega)}$ is the classical projection operator onto the nonempty closed convex subset $\mathcal{K}^{1}(\Omega)$ of $\HH^{1}_{\mathrm{D}}(\Omega)$ for the scalar product~$\dual{\cdot}{\cdot}_{\HH^{1}_{\mathrm{D}}(\Omega)}$.
\end{myProp}
\subsubsection{A Tresca Friction Problem}\label{sectionproblèmedeTresca1ercas}
Finally let us consider the Tresca friction problem given by
\begin{equation}\tag{TP} \label{PbNeumannDirichletTresca}
\arraycolsep=2pt
\left\{
\begin{array}{rcll}
-\Delta u &	= & f   & \text{ in } \Omega , \\
u & = & 0  & \text{ on } \Gamma_{\mathrm{D}} ,\\
\partial_{\nn} u & = & k  & \text{ on } \Gamma_{\mathrm{N}} ,\\
|\partial_{\nn}u|\leq g \text{ and } u\partial_{\nn}u +g|u| & = & 0  & \text{ on } \Gamma_{\mathrm{T}}.
\end{array}
\right.
\end{equation}
We assume that almost every point of $\Gamma_{\mathrm{T}}$ are in $\mathrm{int}_{\Gamma}({\Gamma_{\mathrm{T}}})$ (see Remark~\ref{rmqhyptresca} for details). The difficulty of the above Tresca friction problem lies on the nonsmooth map $\left|\cdot\right|$ on the boundary~$\Gamma_{\mathrm{T}}$ which generates nonsmooth terms in the weak formulation~\eqref{FormuFaibleDNT} below. Therefore, to get the existence/uniqueness of the weak solution, we are led to use the notion of proximal operator from convex analysis (see Definition~\ref{proxi}).
\begin{myDefn}[Strong solution to the Tresca friction problem]
A strong solution to the Tresca friction problem~\eqref{PbNeumannDirichletTresca} is a function $u\in \HH^{1}(\Omega)$ such that $-\Delta u=f$ in $\mathcal{D}'(\Omega)$, $u=0$ almost everywhere on~$\Gamma_{\mathrm{D}}$, $\partial_{\nn}u\in\mathrm{L}^{2}(\Gamma_{\mathrm{N}}\cup\Gamma_{\mathrm{T}})$ with  $\partial_{\nn}u=k$ almost everywhere on $\Gamma_{\mathrm{N}}$, $|\partial_{\nn}u(s)|\leq g(s)$ and~$u(s)\partial_{\nn}u(s)+g(s)|u(s)|=0$ for almost all $s\in \Gamma_{\mathrm{T}}$.
\end{myDefn}
\begin{myDefn}[Weak solution to the Tresca friction problem]
A weak solution to the Tresca friction problem~\eqref{PbNeumannDirichletTresca} is a function $u\in \HH^{1}_{\mathrm{D}}(\Omega)$ such that
 \begin{equation}\label{FormuFaibleDNT}
     \displaystyle\int_{\Omega} \nabla u\cdot\nabla (v-u)+\int_{\Gamma_{\mathrm{T}}}g|v|-\int_{\Gamma_{\mathrm{T}}}g|u|\geq \int_{\Omega}f(v-u)+\int_{\Gamma_{\mathrm{N}}}k(v-u),\qquad \forall v\in \HH^{1}_{\mathrm{D}}(\Omega).
\end{equation}
\end{myDefn}
\begin{myProp}\label{Trescaequivalartc}
A function $u\in \HH^{1}(\Omega)$ is a strong solution to the Tresca friction problem~\eqref{PbNeumannDirichletTresca} if and only if $u$ is a weak solution to the Tresca friction problem~\eqref{PbNeumannDirichletTresca}.
\end{myProp}

From definition of the proximal operator (see Definition~\ref{proxi}), one deduces the following existence/uniqueness result.
\begin{myProp}\label{existenceunicitePbDNT}
The Tresca friction problem~\eqref{PbNeumannDirichletTresca} admits a unique solution $u\in\HH^{1}_{\mathrm{D}}(\Omega)$ given by
$$
\displaystyle u=\mathrm{prox}_{\phi}(F),
$$
where $F\in\HH^{1}_{\mathrm{D}}(\Omega)$ is the solution to the Dirichlet-Neumann problem~\eqref{PbNeumannDirichlet} with $h=0$ almost everywhere on $\Gamma_{\mathrm{T}}$, and where $\mathrm{prox}_{\phi}$ stands for the proximal operator associated with the Tresca friction functional $\phi$ defined by 
\begin{equation}\label{fonctionnelleTrescaGammaT}
\displaystyle\fonction{\phi}{\HH^{1}_{\mathrm{D}}(\Omega)}{\R}{v}{\displaystyle \phi(v):=\int_{\Gamma_{\mathrm{T}}}g|v| .}
\end{equation}
\end{myProp}
\begin{myRem}\label{rmqhyptresca}
The assumption that almost every point of $\Gamma_{\mathrm{T}}$ are in $\mathrm{int}_{\Gamma}({\Gamma_{\mathrm{T}}})$ is only used to prove that a weak solution to the Tresca friction problem~\eqref{PbNeumannDirichletTresca} is also a strong solution, more precisely to get the Tresca's friction law on~$\Gamma_{\mathrm{T}}$. Of course, some sets do not satisfy this assumption (for instance the well-known Smith–Volterra–Cantor set (see, e.g,~\cite[Example 6.15 Section 6 Chapter 1]{ALIP})). Nevertheless it is trivially satisfied in most of standard cases found in practice. Furthermore, if this assumption is not satisfied, one can also prove that the weak solution to the Tresca friction problem~\eqref{PbNeumannDirichletTresca} is a strong solution by adding the assumption that $g\in\LL^{\infty}(\Gamma_{\mathrm{T}})$, and by using the isometry between the dual of~$(\LL^{1}(\Gamma_{\mathrm{T}}), \| \cdot  \|_{(\LL^{1}(\Gamma_{\mathrm{T}}))_{g}})$ and $ \LL^{\infty}(\Gamma_{\mathrm{T}})$ (with its standard norm  $\| \cdot \|_{\LL^{\infty}(\Gamma_{\mathrm{T}})}$) where~$\| \cdot \|_{\LL^{1}(\Gamma_{\mathrm{T}})_{g}}$ is the norm defined by
$$
\fonction{\left \| \cdot \right \|_{\LL^{1}(\Gamma_{\mathrm{T}})_{g}}}{\LL^{1}(\Gamma_{\mathrm{T}})}{\R}{v}{\displaystyle\int_{\Gamma_{\mathrm{T}}}g\left| v \right |.}
$$
The details are left to the reader.
\end{myRem}
\subsection{Sensitivity Analysis of the Tresca Friction Problem}\label{section4}
In this section we consider the parameterized Tresca friction problem given by
 \begin{equation}\tag{TP$_{t}$} \label{PbNeumannDirichletTrescaPara}
\arraycolsep=2pt
\left\{
\begin{array}{rcll}
-\Delta u_{t} & = & f_{t}   & \text{ in } \Omega , \\
u_{t} & = & 0  & \text{ on } \Gamma_{\mathrm{D}} ,\\
\partial_{\nn} u_{t} & = & k_{t}  & \text{ on } \Gamma_{\mathrm{N}} ,\\
|\partial_{\nn}u_{t}|\leq g_{t} \text{ and } u_{t}\partial_{\nn}u_{t}+g_{t}|u_{t}| & = & 0  & \text{ on } \Gamma_{\mathrm{T}},
\end{array}
\right.
\end{equation}
where $f_{t}\in\LL^{2}(\Omega)$, $k_{t}\in\LL^{2}(\Gamma_{\mathrm{N}})$ and $g_{t}\in\LL^{2}(\Gamma_{\mathrm{T}})$, for all $t\geq0$. We assume that, for all $t\geq0$, $g_{t}>0$ almost everywhere on $\Gamma_{\mathrm{T}}$, and almost every point of $\Gamma_{\mathrm{T}}$ belongs to $\mathrm{int}_{\Gamma}({\Gamma_{\mathrm{T}}})$.

Proposition~\ref{existenceunicitePbDNT} claims that the solution to the Tresca friction problem~\eqref{PbNeumannDirichletTresca} is related to the solution to the Dirichlet-Neumann problem~\eqref{PbNeumannDirichlet} by the proximal operator. Therefore let us start with the sensitivity analysis of the Dirichlet-Neumann problem~\eqref{PbNeumannDirichlet}, which will be useful for the sensitivity analysis of the Tresca friction problem~\eqref{PbNeumannDirichletTresca}. The following proposition is easily proved using the linearity of the Dirichlet-Neumann problem~\eqref{PbNeumannDirichlet} and Proposition~\ref{existenceunicitéDN}.
\begin{myProp}\label{AnalysesensiDN}
Let $F_{t} \in \HH^{1}_{\mathrm{D}}(\Omega)$ be the unique solution to the parameterized Dirichlet-Neumann problem given by
\begin{equation}\tag{DN$_{t}$} \label{PbNeumannDirichletPara}
\arraycolsep=2pt
\left\{
\begin{array}{rcll}
-\Delta F_{t} & = & f_{t}   & \text{ in } \Omega , \\
F_{t} & = & 0  & \text{ on } \Gamma_{\mathrm{D}} ,\\
\partial_{\nn} F_{t} & = &  k_{t}  & \text{ on } \Gamma_{\mathrm{N}} ,\\
\partial_{\nn} F_{t} & = & 0  & \text{ on } \Gamma_{\mathrm{T}},
\end{array}
\right.
\end{equation}
for all $t\geq0$. If the two conditions
\begin{enumerate}
    \item the map $t\in\mathbb{R}^{+}\mapsto f_{t}\in \mathrm{L}^{2}(\Omega)$ is differentiable at $t=0$, with its derivative denoted by~$f'_{0}\in\LL^{2}(\Omega)$;
    \item the map $t\in\mathbb{R}^{+}\mapsto k_{t}\in \LL^{2}(\Gamma_{\mathrm{N}})$ is differentiable at $t=0$, with its derivative denoted by~$k'_{0}\in\LL^{2}(\Gamma_{\mathrm{N}})$;
\end{enumerate}
are satisfied, then the map $t\in\mathbb{R}^{+}\mapsto F_{t}\in \HH^{1}_{\mathrm{D}}(\Omega)$ is differentiable at $t=0$, and its derivative, denoted by~$F'_{0}\in \HH^{1}_{\mathrm{D}}(\Omega)$, is the unique solution to the Dirichlet-Neumann problem given by
\begin{equation}\tag{DN$_{0}'$}\label{PbNeumannDirichletDerivhomogene}
\arraycolsep=2pt
\left\{
\begin{array}{rcll}
-\Delta F'_{0} & = & f'_{0}   & \text{ in } \Omega , \\
F'_{0} & = & 0  & \text{ on } \Gamma_{\mathrm{D}} ,\\
\partial_{\nn} F'_{0} & = & k'_{0}  & \text{ on } \Gamma_{\mathrm{N}} ,\\
\partial_{\nn} F'_{0} & = & 0  & \text{ on } \Gamma_{\mathrm{T}}.
\end{array}
\right.
\end{equation}
\end{myProp}
\subsubsection{Parameterized Tresca Friction Functional and Twice Epi-Differentiability}\label{sectionfonctiontrescaepideriv}
Let us come back to the parameterized Tresca friction problem~\eqref{PbNeumannDirichletTrescaPara}. The Tresca friction functional, defined in~\eqref{fonctionnelleTrescaGammaT}, depends now on the parameter $t\geq0$. Precisely we are led to define the parameterized Tresca friction functional given by
\begin{equation}\label{fonctionnelledeTrescaparacas2}
\displaystyle\fonction{\Phi}{\mathbb{R}^{+}\times \HH^{1}_{\mathrm{D}}(\Omega)}{\R}{(t,w)}{\displaystyle \Phi(t,w):=\int_{\Gamma_{\mathrm{T}}}g_{t}|w|.}
\end{equation}
Note that, for all $t\geq0$, $\Phi(t,\cdot)$ is a proper lower semi-continuous convex function on $\HH^{1}_{\mathrm{D}}(\Omega)$ and, from Proposition~\ref{existenceunicitePbDNT}, the unique solution to the parameterized Tresca friction problem~\eqref{PbNeumannDirichletTrescaPara} is given by
\begin{equation}\label{uproxGcas2}
   \displaystyle u_{t}=\mathrm{prox}_{\Phi(t,\mathord{\cdot})}(F_{t}),
\end{equation}
where $F_{t}$ is the unique solution to the parameterized Dirichlet-Neumann problem~\eqref{PbNeumannDirichletPara}.

As we can see in Equality~\eqref{uproxGcas2}, the proximal operator~$\mathrm{prox}_{\Phi(t,\mathord{\cdot})}$ depends on the parameter~$t\geq0$. This leads us to use Theorem~\ref{TheoABC2018} (extracted from~\cite{8AB}) which characterizes the derivative of a map given by a parameterized proximal operator, using the notion of twice epi-differentiability depending on a parameter (see Definition~\ref{epidiffpara}). Let us underline that this is an important difference with the previous paper~\cite{4ABC}, where the proximal operator was associated to a functional that did not depend on the parameter $t\geq0$, therefore the classical notion of twice epi-differentiability introduced by R.T. Rockafellar (see~\cite{Rockafellar}) was sufficient.

Let us prepare the background for the twice epi-differentiability of the parameterized Tresca friction functional defined in~\eqref{fonctionnelledeTrescaparacas2}. More specifically let us start with the characterization of the convex subdifferential of $\Phi(0,\cdot)$ (see Definition~\ref{sousdiff}). To this aim, we introduce an auxiliary problem defined, for all~$u\in\HH^{1}_{\mathrm{D}}(\Omega)$, by
\begin{equation}\tag{AP$_{u}$}\label{PbannexesousdiffDNT}
\arraycolsep=2pt
\left\{
\begin{array}{rl}
-\Delta v  =  0   & \text{ in } \Omega , \\
v  =  0 & \text{ on } \Gamma_{\mathrm{D}} , \\
\partial_{\nn}v  =  0 & \text{ on } \Gamma_{\mathrm{N}} , \\
\partial_{\nn}v(s) \in  g_{0}(s)\partial |\mathord{\cdot}|(u(s))  & \text{ on } \Gamma_{\mathrm{T}},\\
\end{array}
\right.
\end{equation}
where, for almost all $s\in\Gamma_{\mathrm{T}}$, $\partial|\mathord{\cdot}|(u(s))$ stands for the convex subdifferential of the classical absolute value map $\left|\cdot\right|  :  \mathbb{R}\rightarrow\mathbb{R}$ at $u(s)\in\R$.
For a given $u\in \HH^{1}_{\mathrm{D}}(\Omega)$, a solution to this problem~\eqref{PbannexesousdiffDNT} is a function~$v\in\HH^{1}(\Omega)$ such that $-\Delta v =0$ in~$\mathcal{D}
'(\Omega)$, $v=0$ almost everywhere on $\Gamma_{\mathrm{D}}$, and also with~$\partial_{\nn}v\in\mathrm{L}^{2}(\Gamma_{\mathrm{N}}\cup\Gamma_{\mathrm{T}})$, $\partial_{\nn}v=0$ almost everywhere on $\Gamma_{\mathrm{N}}$, and $\partial_{\nn}v(s)\in g_{0}(s)\partial |\mathord{\cdot} |(u(s))$ for almost all $s\in\Gamma_{\mathrm{T}}$.
\begin{myLem}\label{lemmeAnnexeSousdiff}
Let $u\in \HH^{1}_{\mathrm{D}}(\Omega)$. Then
\begin{center}
    $\partial \Phi(0,\cdot)(u)=$ the set of solutions to Problem~\eqref{PbannexesousdiffDNT}.
\end{center}
\end{myLem}
\begin{proof}
Let $u\in \HH^{1}_{\mathrm{D}}(\Omega)$ and let us prove the two inclusions. Firstly, let $v\in\HH^{1}(\Omega)$ be a solution to Problem~\eqref{PbannexesousdiffDNT}. Then $v\in\HH^{1}_{\mathrm{D}}(\Omega)$, $\partial_{\nn}v\in\LL^{2}(\Gamma_{\mathrm{N}}\cup\Gamma_{\mathrm{T}})$ and~$\partial_{\nn}v(s)\in$ $g_{0}(s)\partial|\mathord{\cdot} |(u(s))$ for almost all~$s\in\Gamma_{\mathrm{T}}$. Hence one has
$$
\displaystyle \partial_{\nn}v(s)(\varphi(s)-u(s))\leq g_{0}(s)(|\varphi(s)|-|u(s)|),
$$
for all $\varphi\in \HH^{1}_{\mathrm{D}}(\Omega)$ and for almost all $s\in\Gamma_{\mathrm{T}}$. It follows that
$$
\displaystyle\int_{\Gamma_{\mathrm{T}}}\partial_{\nn}v(\varphi-u)\leq\int_{\Gamma_{\mathrm{T}}}g_{0}(|\varphi|-|u|),
$$
for all $\varphi\in\HH^{1}_{\mathrm{D}}(\Omega)$. Moreover $-\Delta v=0$ in $\MD'(\Omega)$, thus it holds $-\Delta v=0$ in $\LL^{2}(\Omega)$. Hence, from Green formula (see Proposition~\ref{Green}), one gets
$$
\displaystyle\int_{\Omega}\nabla v\cdot\nabla(\varphi-u)=\dual{\partial_{\nn}v}{\varphi-u}_{\HH^{-1/2}(\Gamma)\times \HH^{1/2}(\Gamma)},
$$
for all $\varphi\in\HH^{1}_{\mathrm{D}}(\Omega)$. Furthermore, for all $\varphi\in\HH^{1}_{\mathrm{D}}(\Omega)$, $\varphi\in\HH^{1/2}_{00}(\Gamma_{\mathrm{N}}\cup\Gamma_{\mathrm{T}})$ which is a vector subspace of $\HH^{1/2}(\Gamma)$. Therefore one has
$$
\displaystyle\dual{\partial_{\nn}v}{\varphi-u}_{\HH^{-1/2}(\Gamma)\times \HH^{1/2}(\Gamma)}=\dual{ \partial_{\nn}v}{\varphi-u}_{\HH^{-1/2}_{00}(\Gamma_{\mathrm{N}}\cup\Gamma_{\mathrm{T}})\times \HH^{1/2}_{00}(\Gamma_{\mathrm{N}}\cup\Gamma_{\mathrm{T}})},
$$
for all $\varphi\in\HH^{1/2}_{00}(\Gamma_{\mathrm{N}}\cup\Gamma_{\mathrm{T}})$. Since $\partial_{\nn}v\in \mathrm{L}^{2}(\Gamma_{\mathrm{N}}\cup\Gamma_{\mathrm{T}})$ and $\partial_{\nn}v=0$ almost everywhere on $\Gamma_{\mathrm{N}}$, this leads to
$$
\displaystyle
\dual{ \partial_{\nn}v}{\varphi-u}_{\HH^{-1/2}_{00}(\Gamma_{\mathrm{N}}\cup\Gamma_{\mathrm{T}})\times \HH^{1/2}_{00}(\Gamma_{\mathrm{N}}\cup\Gamma_{\mathrm{T}})}=\int_{\Gamma_{\mathrm{T}}}\partial_{\nn}v(\varphi-u),
$$
for all $\varphi\in\HH^{1}_{\mathrm{D}}(\Omega)$. Therefore one deduces
$$
\displaystyle\int_{\Omega}\nabla v\cdot\nabla(\varphi-u)\leq
\int_{\Gamma_{\mathrm{T}}}g_{0}(|\varphi|-|u|),
$$
for all $\varphi\in\HH^{1}_{\mathrm{D}}(\Omega)$, that is
$$
\dual{v}{\varphi-u}_{\HH^{1}_{\mathrm{D}}(\Omega)}\leq \Phi(0,\varphi)-\Phi(0,u),
$$
for all $\varphi\in\HH^{1}_{\mathrm{D}}(\Omega)$. Thus $v\in\partial\Phi(0,\cdot)(u)$ and the first inclusion is proved. Conversely let $v\in\partial\Phi(0,\cdot)(u)$. One has
\begin{equation}\label{inegalitesousdiffDNT}
\displaystyle\int_{\Omega}\nabla v\cdot\nabla(\varphi-u)\leq
\int_{\Gamma_{\mathrm{T}}}g_{0}(|\varphi|-|u|),
\end{equation}
for all $\varphi\in\HH^{1}_{\mathrm{D}}(\Omega)$. Considering the function $\varphi=u\pm\psi\in\HH^{1}_{\mathrm{D}}(\Omega)$ with any function $\psi\in\MD(\Omega)$, one deduces from Inequality~\eqref{inegalitesousdiffDNT} that $-\Delta v=0$ in $\MD'(\Omega)$, thus $-\Delta v=0$ in $\LL^{2}(\Omega)$. Hence, from Green formula and Inequality~\eqref{inegalitesousdiffDNT}, it follows that
$$
\dual{\partial_{\nn}v}{\varphi-u}_{\HH^{-1/2}_{00}(\Gamma_{\mathrm{N}}\cup\Gamma_{\mathrm{T}})\times \HH^{1/2}_{00}(\Gamma_{\mathrm{N}}\cup\Gamma_{\mathrm{T}})}\leq
\int_{\Gamma_{\mathrm{T}}}g_{0}(|\varphi|-|u|),
$$
for all $\varphi\in\HH^{1}_{\mathrm{D}}(\Omega)$, and thus also for all $\varphi\in\HH^{1/2}_{00}(\Gamma_{\mathrm{N}}\cup\Gamma_{\mathrm{T}})$. Now let us consider the func\-tion~$\varphi=u+w\in\HH^{1/2}_{00}(\Gamma_{\mathrm{N}}\cup\Gamma_{\mathrm{T}})$ for any $w\in\HH^{1/2}_{00}(\Gamma_{\mathrm{N}}\cup\Gamma_{\mathrm{T}})$. One gets
$$
\left|\dual{\partial_{\nn}v}{w}_{\HH^{-1/2}_{00}(\Gamma_{\mathrm{N}}\cup\Gamma_{\mathrm{T}})\times \HH^{1/2}_{00}(\Gamma_{\mathrm{N}}\cup\Gamma_{\mathrm{T}})}\right|\leq \int_{\Gamma_{\mathrm{T}}}g_{0}|w|\leq\left \| g_{0} \right \|_{\LL^{2}(\Gamma_{\mathrm{T}})}\left \| w \right \|_{\LL^{2}(\Gamma_{\mathrm{N}}\cup\Gamma_{\mathrm{T}})}.
$$
From Proposition~\ref{Ident}, one deduces that $\partial_{\nn}v\in \mathrm{L}^{2}(\Gamma_{\mathrm{N}}\cup\Gamma_{\mathrm{T}})$ and also that
\begin{equation}\label{inegalitesousdiffDNTsensi}
\displaystyle\int_{\Gamma_{\mathrm{N}}\cup\Gamma_{\mathrm{T}}}w\partial_{\nn}v\leq
\int_{\Gamma_{\mathrm{T}}}g_{0}(|u+w|-|u|), 
\end{equation}
for all $w\in\HH^{1/2}_{00}(\Gamma_{\mathrm{N}}\cup\Gamma_{\mathrm{T}})$,
and thus by density for all $w\in\LL^{2}(\Gamma_{\mathrm{N}}\cup\Gamma_{\mathrm{T}})$.
By considering the function~$w\in\LL^{2}(\Gamma_{\mathrm{N}}\cup\Gamma_{\mathrm{T}})$ defined by
$$w:=
\left\{
\begin{array}{rl}
\pm\psi	 & \text{ on } \Gamma_{\mathrm{N}}, \\
0 	 & \text{ on } \Gamma_{\mathrm{T}},
\end{array}
\right.
$$
where $\psi$ is any function in $\LL^{2}(\Gamma_{\mathrm{N}})$, one gets in Inequality~\eqref{inegalitesousdiffDNTsensi} that
$$
\int_{\Gamma_{\mathrm{N}}}\psi\partial_{\nn}v=0,
$$
for all $\psi\in\LL^{2}(\Gamma_{\mathrm{N}})$. Hence $\partial_{\nn}v=0$ almost everywhere on $\Gamma_{\mathrm{N}}$, and Inequality~\eqref{inegalitesousdiffDNTsensi} becomes
\begin{equation}\label{secondeinegalite}
\displaystyle\int_{\Gamma_{\mathrm{T}}}\partial_{\nn}v(\varphi-u)\leq
\int_{\Gamma_{\mathrm{T}}}g_{0}(|\varphi|-|u|),
\end{equation}
for all $\varphi\in\LL^{2}(\Gamma_{\mathrm{N}}\cup\Gamma_{\mathrm{T}})$. Now let $s\in\Gamma_{\mathrm{T}}$ be a Lebesgue point of $\partial_{\nn}v\in\LL^{2}(\Gamma_{\mathrm{N}}\cup\Gamma_{\mathrm{T}})$,  $u\partial_{\nn}v\in\LL^{1}(\Gamma_{\mathrm{N}}\cup\Gamma_{\mathrm{T}})$, $g_{0}\in\LL^{2}(\Gamma_{\mathrm{T}})$ and of~$g_{0}|u|\in\LL^{1}(\Gamma_{\mathrm{T}})$, such that $s\in\mathrm{int}_{\Gamma}({\Gamma_{\mathrm{T}}})$. Let us consider the function $\varphi\in \mathrm{L}^{2}(\Gamma_{\mathrm{N}}\cup\Gamma_{\mathrm{T}})$ defined by
$$\varphi:=
\left\{
\begin{array}{rl}
x   & \text{ on } B_{\Gamma}(s,\varepsilon) , \\
u   & \text{ on } \Gamma_{\mathrm{N}}\cup\Gamma_{\mathrm{T}}\textbackslash B_{\Gamma}(s,\varepsilon) ,
\end{array}
\right.
$$
with $x\in\R$ and $\varepsilon>0$ such that $B_{\Gamma}(s,\varepsilon)\subset\Gamma_{\mathrm{T}}$. Then one has from Inequality~\eqref{secondeinegalite} that
$$
\displaystyle \frac{1}{\left|B_{\Gamma}(s,\varepsilon)\right|}\int_{B_{\Gamma}(s,\varepsilon)}\partial_{\nn}v(x-u)\leq\frac{1}{\left|B_{\Gamma}(s,\varepsilon)\right|}\int_{B_{\Gamma}(s,\varepsilon)}g_{0}|x|-\frac{1}{\left|B_{\Gamma}(s,\varepsilon)\right|}\int_{B_{\Gamma}(s,\varepsilon)}g_{0}|u|,
$$
thus $\partial_{\nn}v(s)(x-u(s))\leq g_{0}(s)(|x|-|u(s)|)$ by letting $\varepsilon\rightarrow0^{+}$. This inequality is true for any~$x \in \R$, therefore $\partial_{\nn}v(s)\in g_{0}(s)\partial|\mathord{\cdot} |(u(s)).
$
Moreover, since almost every point of $\Gamma_{\mathrm{T}}$ are in~$\mathrm{int}_{\Gamma}({\Gamma_{\mathrm{T}}})$ and are Lesbegue points of $\partial_{\nn}v\in\LL^{2}(\Gamma_{\mathrm{N}}\cup\Gamma_{\mathrm{T}})$, $u\partial_{\nn}v\in\LL^{1}(\Gamma_{\mathrm{N}}\cup\Gamma_{\mathrm{T}})$, $g_{0}\in\LL^{2}(\Gamma_{\mathrm{T}})$ and of $g_{0}|u|\in\LL^{1}(\Gamma_{\mathrm{T}})$, one deduces 
$$
\partial_{\nn}v(s)\in g_{0}(s)\partial|\mathord{\cdot} |(u(s)),
$$
for almost all $s\in\Gamma_{\mathrm{T}}$, and this proves the second inclusion.
\end{proof}

The twice epi-differentiability is defined using the second-order difference quotient functions. Therefore let us compute the following second-order difference quotient functions of $\Phi$ at $u\in\HH^{1}_{\mathrm{D}}(\Omega)$ for $v\in\partial\Phi(0,\cdot)(u)$ defined by
$$
\fonction{\Delta_{t}^{2}\Phi(u|v)}{\HH^1_{\mathrm{D}}(\Omega)}{\R}{w}{ \displaystyle\Delta_{t}^{2}\Phi(u|v)(w):=\frac{\Phi(t,u+t w)-\Phi(t,u)-t\dual{ v}{w}_{\HH^{1}_{\mathrm{D}}(\Omega)}}{t^{2}},}
$$
for all $t>0$.
\begin{myProp}\label{epidiffoffunctionG}
For all $t>0$, $u\in \HH^{1}_{\mathrm{D}}(\Omega)$ and $v\in\partial\Phi(0,\cdot)(u)$, it holds that
\begin{equation}\label{Delta2}
      \displaystyle\Delta_{t}^{2}\Phi(u|v)(w)=\int_{\Gamma_{\mathrm{T}}}\Delta_{t}^{2}G(s)(u(s)|\partial_{\nn}v(s))(w(s))\mathrm{d}s,
\end{equation}
for all $w\in \HH^{1}_{\mathrm{D}}(\Omega)$, where, for almost all $s\in\Gamma_{\mathrm{T}}$, $\Delta_{t}^{2}G(s)(u(s)|\partial_{\nn}v(s))$ stands for the second-order difference quotient function of $G(s)$ at $u(s)\in\R$ for $\partial_{\nn}v(s) \in g_{0}(s)\partial|\mathord{\cdot} |(u(s))$, with~$G(s)$ defined by
$$ 
\fonction{G(s)}{\mathbb{R}^{+}\times\mathbb{R}}{\R}{(t,x)}{G(s)(t,x):=g_{t}(s)|x|.}
$$
\end{myProp}
\begin{myRem}
Note that, for almost all $s\in\Gamma_{\mathrm{T}}$ and all $t\geq 0$, $G(s)(t,\cdot):=g_{t}(s)|\mathord{\cdot} |$ is a proper lower semi-continuous convex function on $\mathbb{R}$. Moreover, since $g_{0}>0$ almost everywhere on~$\Gamma_{\mathrm{T}}$, it follows that
$$
\partial\left[G(s)(0,\mathord{\cdot})\right](x)=g_{0}(s)\partial |\mathord{\cdot} |(x),
$$
for all $x\in\R$ and for almost all $s\in\Gamma_{\mathrm{T}}$.
\end{myRem}
\begin{proof}[Proof of Proposition~\ref{epidiffoffunctionG}]
Let $t>0$, $u\in \HH^{1}_{\mathrm{D}}(\Omega)$ and $v\in\partial\Phi(0,\cdot)(u)$. From Lemma~\ref{lemmeAnnexeSousdiff} and Green formula (see Proposition~\ref{Green}), one deduces
$$
\displaystyle \dual{ v}{w}_{\HH^{1}_{\mathrm{D}}(\Omega)}=\dual{ \partial_{\nn}v}{w}_{\HH^{-1/2}(\Gamma)\times \HH^{1/2}(\Gamma)},
$$
for all $w\in\HH^{1}_{\mathrm{D}}(\Omega)$. Moreover, similarly to Lemma~\ref{lemmeAnnexeSousdiff}, one gets
$$
\displaystyle \dual{ v}{w}_{\HH^{1}_{\mathrm{D}}(\Omega)}=\int_{\Gamma_{\mathrm{T}}}w\partial_{\nn}v,
$$
for all $w\in\HH^{1}_{\mathrm{D}}(\Omega)$. Thus it follows that
$$
        \displaystyle \Delta_{t}^{2}\Phi(u|v)(w)=\int_{\Gamma_{\mathrm{T}}}\frac{g_{t}(s)|u(s)+t w(s)|-g_{t}(s)|u(s)|-tw(s)\partial_{\nn}v(s)}{t^{2}} \mathrm{d}s,
$$
for all $w\in\HH^{1}_{\mathrm{D}}(\Omega)$. Furthermore, since $\partial_{\nn}v(s)\in g_{0}(s)\partial |\mathord{\cdot} |(u(s))$ for almost all $s\in\Gamma_{\mathrm{T}}$, one deduces that
$$
\displaystyle \Delta_{t}^{2}\Phi(u|v)(w)=\int_{\Gamma_{\mathrm{T}}}\Delta_{t}^{2}G(s)(u(s)|\partial_{\nn}v(s))(w(s))\mathrm{d}s,
$$
for all $w\in \HH^{1}_{\mathrm{D}}(\Omega)$, which concludes the proof.
\end{proof}

From the above proposition we note that the twice epi-differentiability of the parameterized Tresca friction functional is strongly related to the twice epi-differentiability of the function $G(s)$ for almost all~$s\in\Gamma_{\mathrm{T}}$. Therefore the computation of the second-order epi-derivative of $G(s)$ for almost all~$s\in\Gamma_{\mathrm{T}}$ is the next step.
\begin{myProp}\label{épidiffgabs}
Assume that, for almost all $s\in\Gamma_{\mathrm{T}}$, the map $t\in\mathbb{R}^{+}\mapsto g_{t}(s)\in\mathbb{R}^{+}$ is differentiable at $t=0$, with its derivative denoted by $g'_{0}(s)$. Then, for almost all $s\in\Gamma_{\mathrm{T}}$, the map~$G(s)$ is twice epi-differentiable at any~$x\in\mathbb{R}$ and for all $y\in g_{0}(s)\partial |\mathord{\cdot} |(x)$ with
$$
\displaystyle \mathrm{D}_{e}^{2}G(s)(x|y)(z)=\mathrm{I}_{\mathrm{K}_{ x,\frac{y}{g_{0}(s)}}}(z)+g'_{0}(s)\frac{y}{g_{0}(s)}z,
$$
for all $z\in\mathbb{R}$, where $\mathrm{I}_{\mathrm{K}_{ x,\frac{y}{g_{0}(s)}}}$ stands for the indicator function of the set $\mathrm{K}_{ x,\frac{y}{g_{0}(s)}}$ (see Example~\ref{epidiffabs}).
\end{myProp}
\begin{proof}
We use the same notations as in Definitions~\ref{epidiff} and~\ref{epidiffpara}.
Let $x\in\mathbb{R}$. Then, for almost all~$s\in\Gamma_{\mathrm{T}}$, for all $y\in g_{0}(s)\partial|\mathord{\cdot} |(x)$ and all $z\in\mathbb{R}$, one has
\begin{multline*}
\displaystyle\Delta_{t}^{2}G(s)(x|y)(z)=\frac{g_{t}(s)|x+t z|-g_{t}(s)|x|-tyz}{t^{2}}\\=g_{t}(s)\frac{|x+t z|-|x|-t \frac{y}{g_{0}(s)}z}{t^{2}}+\frac{\left(g_{t}(s)-g_{0}(s)\right)y}{tg_{0}(s)}z,
\end{multline*}
that is
$$
\displaystyle \Delta_{t}^{2}G(s)(x|y)(z)=g_{t}(s)\delta_{t}^{2}|\mathord{\cdot} | \left( x|\frac{y}{g_{0}(s)} \right)(z)+\frac{\left(g_{t}(s)-g_{0}(s)\right)y}{tg_{0}(s)}z,
$$ 
with $\frac{y}{g_{0}(s)}\in\partial|\mathord{\cdot} |(x)$, and where $\delta_{t}^{2}|\mathord{\cdot} |(x|\frac{y}{g_{0}(s)})$ is the second-order difference quotient function of~$|\mathord{\cdot}|$ at $x$ for $\frac{y}{g_{0}(s)}$ (see Definition~\ref{epidiff} since $|\mathord{\cdot}|$ is $t$-independent function). Using the characterization of Mosco epi-convergence (see Proposition~\ref{caractMosco}) and Example~\ref{epidiffabs}, it follows that the map~$G(s)$ is twice epi-differentiable at~$x$ for $y$ with
$$
\displaystyle \mathrm{D}_{e}^{2}G(s)(x|y)(z)=\mathrm{I}_{\mathrm{K}_{ x,\frac{y}{g_{0}(s)}}}(z)+g'_{0}(s)\frac{y}{g_{0}(s)}z,
$$
for all $z\in\mathbb{R}$, and the proof is complete.
\end{proof}
\begin{myRem}
The perturbation of the friction term~$g_t$ in the parameterized Tresca friction problem~\eqref{PbNeumannDirichletTrescaPara} (which is not considered in the previous paper~\cite{4ABC}) generates an additional term in the expression of the second-order epi-derivative of~$G(s)$, for almost all $s\in\Gamma_{\mathrm{T}}$, at all $x \in \R$ for all~$y\in g_{0}(s)\partial |\mathord{\cdot} |(x)$, given by the function~$z\in\R\mapsto g'_{0}(s)\frac{y}{g_{0}(s)}z\in\R$.
\end{myRem}
\subsubsection{The Derivative of the Solution to the Parameterized Tresca Friction Problem}
From the previous results and some additional assumptions detailed below, we are now in a position to state and prove the main result of this paper that characterizes the derivative of the solution to the parameterized Tresca friction problem~\eqref{PbNeumannDirichletTrescaPara}.
\begin{myTheorem}\label{caractu0derivDNT}
Let $u_{t}\in\HH^{1}_{\mathrm{D}}(\Omega)$ be the unique solution to the parameterized Tresca friction problem~\eqref{PbNeumannDirichletTrescaPara} for all~$t \geq 0$. Assume that
\begin{enumerate}
    \item the map $t\in\mathbb{R}^{+}\mapsto f_{t}\in \mathrm{L}^{2}(\Omega)$ is differentiable at $t=0$, with its derivative denoted by~$f'_{0}\in\LL^{2}(\Omega)$;\label{hypo1}
    \item the map $t\in\mathbb{R}^{+}\mapsto k_{t}\in \LL^{2}(\Gamma_{\mathrm{N}})$ is differentiable at $t=0$, with its derivative denoted by~$k'_{0}\in\LL^{2}(\Gamma_{\mathrm{N}})$;\label{hypo2}
    \item for almost all $s\in\Gamma_{\mathrm{T}}$, the map $t\in\mathbb{R}^{+}\mapsto g_{t}(s)\in\mathbb{R}^{+}$ is differentiable at $t=0$, with its derivative denoted by $g'_{0}(s)$, and also $g_{0}'\in \mathrm{L}^{2}(\Gamma_{\mathrm{T}})$;\label{hypo3}
    \item the parameterized Tresca friction functional $\Phi$ defined in~\eqref{fonctionnelledeTrescaparacas2} is twice epi-differentiable (see Definition~\ref{epidiffpara}) at $u_{0}$ for $F_{0}-u_{0}\in\partial \Phi(0,\cdot)(u_{0})$, with\label{hypo4}
\begin{equation}\label{hypoth1}
\displaystyle\mathrm{D}_{e}^{2}\Phi(u_{0}|F_{0}-u_{0})(w)=\int_{\Gamma_{\mathrm{T}}}\mathrm{D}_{e}^{2}G(s)(u_{0}(s)|\partial_{\nn}(F_{0}-u_{0})(s))(w(s))\mathrm{d}s,
\end{equation}    
for all $w\in \HH^{1}_{\mathrm{D}}(\Omega)$, where $F_{0}$ is the unique solution to the parameterized Dirichlet-Neumann problem~\eqref{PbNeumannDirichletPara} for the parameter~$t=0$.
\end{enumerate}
Then the map $t\in\mathbb{R}^{+}\mapsto u_{t}\in\HH^{1}_{\mathrm{D}}(\Omega)$ is differentiable at $t=0$, and its derivative denoted by~$u'_{0}\in\HH^{1}_{\mathrm{D}}(\Omega)$ is the unique weak solution to the Signorini problem
\begin{equation}\tag{SP$_{0}'$}\label{caractu0DNT}
\arraycolsep=2pt
\left\{
\begin{array}{rcll}
-\Delta u_{0}' & = & f_{0}'  & \text{ in } \Omega , \\
u_{0}' & = & 0  & \text{ on } \Gamma_{\mathrm{D}}\cup\Gamma^{u_{0},g_{0}}_{\mathrm{T}_{\mathrm{S}_{\mathrm{D}}}}, \\
\partial_{\nn} u_{0}' & = & k_{0}'  & \text{ on } \Gamma_{\mathrm{N}}, \\
\partial_{\nn} u_{0}' & = & g_{0}'\frac{\partial_{\nn}u_{0}}{g_{0}}  & \text{ on } \Gamma^{u_{0},g_{0}}_{\mathrm{T}_{\mathrm{S}_{\mathrm{N}}}}, \\
 u_{0}'\leq0\text{, } \partial_{\nn}u_{0}'\leq g_{0}'\frac{\partial_{\nn}u_{0}}{g_{0}} \text{ and } u_{0}'\left(\partial_{\nn}u_{0}'- g_{0}'\frac{\partial_{\nn}u_{0}}{g_{0}}\right) & = & 0  & \text{ on } \Gamma^{u_{0},g_{0}}_{\mathrm{T}_{\mathrm{S-}}}, \\
 u_{0}'\geq0\text{, } \partial_{\nn}u_{0}'\geq g_{0}'\frac{\partial_{\nn}u_{0}}{g_{0}} \text{ and } u_{0}'\left(\partial_{\nn}u_{0}'- g_{0}'\frac{\partial_{\nn}u_{0}}{g_{0}}\right) & = & 0  & \text{ on } \Gamma^{u_{0},g_{0}}_{\mathrm{T}_{\mathrm{S+}}},
\end{array}
\right.
\end{equation}
where $\Gamma_{\mathrm{T}}=\Gamma^{u_{0},g_{0}}_{\mathrm{T}_{\mathrm{S}_{\mathrm{N}}}}\cup
\Gamma^{u_{0},g_{0}}_{\mathrm{T}_{\mathrm{S}_{\mathrm{D}}}}\cup\Gamma^{u_{0},g_{0}}_{\mathrm{T}_{\mathrm{S-}}}\cup\Gamma^{u_{0},g_{0}}_{\mathrm{T}_{\mathrm{S+}}}$ with
$$
\begin{array}{l}
\Gamma^{u_{0},g_{0}}_{\mathrm{T}_{\mathrm{S}_{\mathrm{N}}}}:=\left\{s\in\Gamma_{\mathrm{T}} \mid  u_{0}(s)\neq0\right \}, \\
\Gamma^{u_{0},g_{0}}_{\mathrm{T}_{\mathrm{S}_{\mathrm{D}}}}:=\left\{s\in\Gamma_{\mathrm{T}} \mid  u_{0}(s)=0 \text{ and } \partial_{\nn}u_{0}(s)\in\left(-g_{0}(s),g_{0}(s)\right)\right \}, \\
\Gamma^{u_{0},g_{0}}_{\mathrm{T}_{\mathrm{S-}}}:=\left\{s\in\Gamma_{\mathrm{T}} \mid u_{0}(s)=0 \text{ and } \partial_{\nn}u_{0}(s)=g_{0}(s)\right \}, \\
\Gamma^{u_{0},g_{0}}_{\mathrm{T}_{\mathrm{S+}}}:=\left\{s\in\Gamma_{\mathrm{T}} \mid  u_{0}(s)=0 \text{ and } \partial_{\nn}u_{0}(s)=-g_{0}(s)\right \}.
\end{array}
$$
\end{myTheorem}
\begin{myRem}\label{Remarquenotwice}
   From Proposition~\ref{epidiffoffunctionG}, one can naturally expect that the second-order epi-derivative of the parameterized Tresca friction functional $\Phi$ at $u_{0}$ for $F_{0}-u_{0}$ is given by Equality~\eqref{hypoth1}, which corresponds to the inversion of symbols $\mathrm{ME}\text{-}\mathrm{lim}$ and $\int_{\Gamma_{\mathrm{T}}}$  in Equality~\eqref{Delta2}. Nevertheless, to the best of our knowledge, this inversion is an open question. Therefore we do not know, in general, if the parameterized Tresca friction functional is indeed twice epi-differentiable at $u_{0}$ for $F_{0}-u_{0}$. Nevertheless, in Appendix~\ref{casparticulier}, we prove it in several particular cases corresponding to practical situations.
\end{myRem}
\begin{myRem}\label{Remarquepbbienpose}
 The problem~\eqref{caractu0DNT} in Theorem~\ref{caractu0derivDNT} is a well-posed problem since
 $$
 \left|\frac{\partial_{\nn}u_{0}(s)}{g_{0}(s)}\right|\leq1,
 $$
 for almost all $s\in\Gamma_{\mathrm{T}}$, and hence $g_{0}' \frac{\partial_{\nn}u_{0}}{g_{0}}\in \mathrm{L}^{2}(\Gamma_{\mathrm{T}})$ since $g'_{0}\in \LL^{2}(\Gamma_{\mathrm{T}})$.
\end{myRem} 
\begin{myRem}
Consider the framework of Theorem~\ref{caractu0derivDNT}. Note that $u'_{0}$ is the unique weak solution to the Signorini problem~\eqref{caractu0DNT}, but is not necessarily a strong solution. Nevertheless, in the case where $\partial_{\nn}u'_{0}\in \mathrm{L}^{2}(\Gamma_{\mathrm{N}}\cup\Gamma_{\mathrm{T}})$ and the decomposition~$\Gamma=\Gamma_{\mathrm{D}}\cup\Gamma_{\mathrm{N}}\cup\Gamma^{u_{0},g_{0}}_{\mathrm{T}_{\mathrm{S}_{\mathrm{N}}}}\cup
\Gamma^{u_{0},g_{0}}_{\mathrm{T}_{\mathrm{S}_{\mathrm{D}}}}\cup\Gamma^{u_{0},g_{0}}_{\mathrm{T}_{\mathrm{S-}}}\cup\Gamma^{u_{0},g_{0}}_{\mathrm{T}_{\mathrm{S+}}}$ is consistent (see Definition~\ref{regulieresens2}), then~$u'_{0}$ is a strong solution to the Signorini problem~\eqref{caractu0DNT}.
\end{myRem}
\begin{proof}[Proof of Theorem~\ref{caractu0derivDNT}]
From Hypothesis~\ref{hypo4} and Proposition~\ref{épidiffgabs}, it follows that
\begin{equation}\label{otherepi}
\displaystyle \mathrm{D}_{e}^{2}\Phi(u_{0}|F_{0}-u_{0})(w)=\mathrm{I}_{\mathcal{K}_{u_{0},\frac{\partial_{\nn}(F_{0}-u_{0})}{g_{0}}}}(w)+\int_{\Gamma_{\mathrm{T}}}g'_{0}(s)\frac{\partial_{\nn}(F_{0}-u_{0})(s)}{g_{0}(s)}w(s)\mathrm{d}s,
\end{equation}
for all $w\in \HH^{1}_{\mathrm{D}}(\Omega)$, where $\mathcal{K}_{u_{0},\frac{\partial_{\nn}(F_{0}-u_{0})}{g_{0}}}$ is the nonempty closed convex subset of $\HH^{1}_{\mathrm{D}}(\Omega)$ defined by 
$$
\mathcal{K}_{u_{0},\frac{\partial_{\nn}(F_{0}-u_{0})}{g_{0}}}:=\left\{ w\in \HH^{1}_{\mathrm{D}}(\Omega)\mid w(s)\in \mathrm{K}_{u_{0}(s),\frac{\partial_{\nn}(F_{0}-u_{0})(s)}{\scriptstyle{g_{0}(s)}}} \text{ for almost all }s\in\Gamma_{\mathrm{T}} \right\}.
$$
Moreover, since $\partial_{\nn}F_{0}=0$ on $\Gamma_{\mathrm{T}}$, one gets
\begin{multline*}
\displaystyle \mathcal{K}_{u_{0},\frac{\partial_{\nn}(F_{0}-u_{0})}{g_{0}}}
\\ 
=\left\{ w\in \HH^{1}_{\mathrm{D}}(\Omega)\mid w\leq 0 \text{ a.e.\ on }\Gamma^{u_{0},g_{0}}_{\mathrm{T}_{\mathrm{S-}}}\text{, } w\geq 0 \text{ a.e.\ on }\Gamma^{u_{0},g_{0}}_{\mathrm{T}_{\mathrm{S+}}}\text{, } w=0 \text{ a.e.\ on } \Gamma^{u_{0},g_{0}}_{\mathrm{T}_{\mathrm{S}_{\mathrm{D}}}} \right\},
\end{multline*}
where subsets
$\Gamma^{u_{0},g_{0}}_{\mathrm{T}_{\mathrm{S}_{\mathrm{N}}}}$,
$\Gamma^{u_{0},g_{0}}_{\mathrm{T}_{\mathrm{S}_{\mathrm{D}}}}$,
$\Gamma^{u_{0},g_{0}}_{\mathrm{T}_{\mathrm{S-}}}$ and $\Gamma^{u_{0},g_{0}}_{\mathrm{T}_{\mathrm{S+}}}$ are defined in Theorem~\ref{caractu0derivDNT}. One can easily prove, using norms equivalence between $\left \| \cdot \right \|_{\HH^{1}_{\mathrm{D}}(\Omega)}$ and $\left \| \cdot \right \|_{\HH^{1}(\Omega)}$ on $\HH^{1}_{\mathrm{D}}(\Omega)$, that $\mathrm{D}_{e}^{2}\Phi(u_{0}|F_{0}-u_{0})$ is a proper lower semi-continuous convex function on $\HH^{1}_{\mathrm{D}}(\Omega)$. Moreover, from Hypotheses~\ref{hypo1} and~\ref{hypo2}, we know from Proposition~\ref{AnalysesensiDN} that the map $t\in\mathbb{R}^{+}\mapsto F_{t}\in \HH^{1}_{\mathrm{D}}(\Omega)$ is differentiable at $t=0$, with its derivative $F'_{0}\in \HH^{1}_{\mathrm{D}}(\Omega)$ being the unique solution to the Dirichlet-Neumann  problem~\eqref{PbNeumannDirichletDerivhomogene}. Thus, using~\eqref{uproxGcas2} and Theorem~\ref{TheoABC2018}, the map $t\in\mathbb{R}^{+}\mapsto u_{t}\in \HH^{1}_{\mathrm{D}}(\Omega)$ is differentiable at $t=0$, and its derivative $u_{0}'\in\HH^{1}_{\mathrm{D}}(\Omega)$ satisfies
$$
\displaystyle u_{0}'=\mathrm{prox}_{\mathrm{D}_{e}^{2}\Phi(u_{0}|F_{0}-u_{0})}(F_{0}'),
$$
which, from the definition of the proximal operator (see Proposition~\ref{proxi}), leads to
$$
\displaystyle F_{0}'-u_{0}'\in\partial \mathrm{D}_{e}^{2}\Phi(u_{0}|F_{0}-u_{0})(u_{0}'),
$$
which means that
$$
\displaystyle \dual{ F_{0}'-u'_{0}}{v-u_{0}'}_{\HH^{1}_{\mathrm{D}}(\Omega)}\leq \mathrm{D}_{e}^{2}\Phi(u_{0}|F_{0}-u_{0})(v) -\mathrm{D}_{e}^{2}\Phi(u_{0}|F_{0}-u_{0})(u_{0}'),
$$
for all $v\in \HH^{1}_{\mathrm{D}}(\Omega)$. Hence we get that
\begin{multline*}
    \int_{\Omega}\nabla \left(F'_{0}-u'_{0}\right)\cdot\nabla(v-u'_{0})\\ \leq \mathrm{I}_{\mathcal{K}_{u_{0},\frac{\partial_{\nn}(F_{0}-u_{0})}{g_{0}}}}(v)-\mathrm{I}_{\mathcal{K}_{u_{0},\frac{\partial_{\nn}(F_{0}-u_{0})}{g_{0}}}}(u_{0}')+\int_{\Gamma_{\mathrm{T}}}g'_{0}(s)\frac{\partial_{\nn}(F_{0}-u_{0})(s)}{g_{0}(s)}(v(s)-u_{0}'(s))\mathrm{d}s,
\end{multline*}
for all $v\in \HH^{1}_{\mathrm{D}}(\Omega)$. Moreover, since $\partial_{\nn}F_{0}=0$ on $\Gamma_{\mathrm{T}}$ and $F'_{0}$ is the unique solution to the Dirichlet-Neumann problem~\eqref{PbNeumannDirichletDerivhomogene}, it follows that
\begin{multline*}
    \int_{\Omega}\nabla u'_{0}\cdot\nabla(v-u'_{0}) \geq \mathrm{I}_{\mathcal{K}_{u_{0},\frac{\partial_{\nn}(F_{0}-u_{0})}{g_{0}}}}(u'_{0})-\mathrm{I}_{\mathcal{K}_{u_{0},\frac{\partial_{\nn}(F_{0}-u_{0})}{g_{0}}}}(v)\\+\int_{\Omega}f_{0}'(v-u'_{0})+\int_{\Gamma_{\mathrm{N}}}k_{0}'(v-u_{0}')+\int_{\Gamma_{\mathrm{T}}}g'_{0}(s)\frac{\partial_{\nn}u_{0}(s)}{g_{0}(s)}(v(s)-u_{0}'(s))\mathrm{d}s,
\end{multline*}
for all $v\in\HH^{1}_{\mathrm{D}}(\Omega)$. Hence $u_{0}'\in\mathcal{K}_{u_{0},\frac{\partial_{\nn}(F_{0}-u_{0})}{g_{0}}}$ and
$$
 \int_{\Omega}\nabla u'_{0}\cdot\nabla(v-u'_{0}) \geq \int_{\Omega}f_{0}'(v-u'_{0})+\int_{\Gamma_{\mathrm{N}}}k_{0}'(v-u_{0}')+\int_{\Gamma_{\mathrm{T}}}g'_{0}(s)\frac{\partial_{\nn}u_{0}(s)}{g_{0}(s)}(v(s)-u_{0}'(s))\mathrm{d}s,
$$
 for all $v\in\mathcal{K}_{u_{0},\frac{\partial_{\nn}(F_{0}-u_{0})}{g_{0}}}$. From the weak formulation of the Signorini problem (see Definition~\ref{WeaksolutionSigno}), one deduces that $u'_{0}$ is the unique weak solution to the Signorini problem~\eqref{caractu0DNT}. The proof is complete.
\end{proof}

Roughly speaking, Theorem~\ref{caractu0derivDNT} claims that the first-order approximation in $\HH^{1}(\Omega)$ of the solution $u_{t}$ to the parameterized Tresca friction problem~\eqref{PbNeumannDirichletTrescaPara} is given by $u_{0}+tu'_{0}$ for small values of $t\geq0$, where~$u'_{0}$ is the solution to the Signorini problem~\eqref{caractu0DNT}. In the next section, we illustrate this comment with some numerical simulations, by comparing $u_{t}$ and $u_{0}+tu'_{0}$ in $\HH^{1}$-norm for small values of~$t\geq0$.

\section{Numerical Simulations}\label{simunum}
 In this section we illustrate Theorem~\ref{caractu0derivDNT} with some numerical simulations. The same notations used in the previous section are preserved and, for an explicit two-dimensional example described in Section~\ref{presentexample}, we compare in $\HH^{1}$-norm the solution~$u_{t}$ to the parameterized Tresca friction problem~\eqref{PbNeumannDirichletTrescaPara} with its first-order approximation $u_{0}+tu'_{0}$ for small values of $t\geq0$, where $u'_{0}$ is the solution to the Signorini problem~\eqref{caractu0DNT}.
 
 Numerical simulations are performed using Freefem++ software (see~\cite{10HECHT}) and iterative switching algorithms (see \cite{11AIT}). The results are presented in Section~\ref{numresults}. In a nutshell, let us recall that those iterative switching algorithms operate by checking at each iteration if the boundary conditions are satisfied, and if they are not, by imposing them and restarting the computation (see~\cite[Annexe C p.25]{4ABC} for detailed explanations on those algorithms). The convergence proof of these algorithms is not established yet but their performance are experimentally validated. Let us emphasize that our aim in this section is not to study these algorithms rigorously but only to illustrate the main result of this paper with a simple and easily implementable method to solve the Signorini problem and the Tresca friction problem. Let us mention that there exist other algorithms, like for instance Nitsche methods (see, e.g.,~\cite{14CHO,CHO2}), hybrid methods (see, e.g.,~\cite{14BEL}), mixed methods (see, e.g.,~\cite{HAS}) and more, that are not used here, but could be more efficient for future researches.
\subsection{Mathematical Framework}\label{presentexample}
In this section we describe the example used for numerical simulations. This example is inspired from the one introduced in the paper~\cite{4ABC}. Let $d=2$ and~$\Omega$ be the unit disk of $\R^{2}$, and assume that the decomposition of the boundary $\Gamma=\partial{\Omega}$ is given by
$$
\Gamma=\Gamma_{\mathrm{D}}\cup\Gamma_{\mathrm{N}}\cup\Gamma_{\mathrm{T}},
$$
with
$$
\begin{array}{l}
\Gamma_{\mathrm{D}}:=\left\{(\cos\theta,\sin\theta)\in\Gamma \mid \frac{\pi}{4}<\theta\leq\frac{\pi}{2}\right\}, \\
\Gamma_{\mathrm{N}}:=\left\{(\cos\theta,\sin\theta)\in\Gamma \mid \frac{\pi}{2}<\theta<\frac{3\pi}{4}\right\}, \\
\Gamma_{\mathrm{T}}:=\left\{(\cos\theta,\sin\theta)\in\Gamma \mid -\frac{5\pi}{4}\leq\theta\leq\frac{\pi}{4}\right\}.
\end{array}
$$
Let $f\in\LL^{2}(\Omega)$ be the function defined by 
$$
\fonction{f}{\Omega}{\R}{(x,y)}{\displaystyle f(x,y):= -2\xi(x)-2x\xi'(x)-\frac{1}{2}(x^{2}+y^{2}-1)\xi''(x) ,}
$$
where $\xi$ is given by
$$
\fonction{\xi}{[-1,1]}{\R}{x}{\displaystyle\xi(x):=\left\{ \begin{array}{rcll}
-1	&     & \text{ if } & -1\leq x\leq-\frac{1}{2} , \\
\sin(\pi x)	&   & \text{ if } & -\frac{1}{2}\leq x\leq \frac{1}{2},\\
1	&     & \text{ if } & \frac{1}{2}\leq x\leq1 .
\end{array} 
\right.}
$$
Let us introduce, for all $t\geq0$, the function $f_{t}\in\LL^{2}(\Omega)$ defined by
$$
\fonction{f_{t}}{\Omega}{\R}{(x,y)}{\displaystyle f_{t}(x,y):=\exp(t)f(x,y),}
$$
the function $g_{t}\in\LL^{2}(\Gamma_{\mathrm{T}})$ defined by 
$$
\fonction{g_{t}}{\Gamma_{\mathrm{T}}}{\R}{(x,y)}{\displaystyle g_{t}(x,y):=1+t,}
$$
and also the function $k_{t}\in\LL^{2}(\Gamma_{\mathrm{N}})$ defined by
$$
\fonction{k_{t}}{\Gamma_{\mathrm{N}}}{\R}{(x,y)}{\displaystyle k_{t}(x,y):=(1+t)\xi(x).}
$$
This choice of functions is justified by the fact that we are able to determinate explicitly the solution $u_{0}$ to the parameterized Tresca friction problem~\eqref{PbNeumannDirichletTrescaPara} for $t=0$, which is given by 
$$u_{0}(x,y)=\displaystyle  \frac{1}{2}\left(x^{2}+y^{2}-1\right)\xi(x),
$$ 
for all $(x,y)\in\Omega$. The knowledge of the solution $u_{0}$ reduces errors due to the approximations for the numerical computations of $u'_{0}$ and $u_{0}+tu'_{0}$. Indeed, since $\partial_{\nn}u=\xi$ and $u_{0}=0$ almost everywhere on $\Gamma$, we can directly express the decomposition
$$
\Gamma_{\mathrm{T}}=\Gamma^{u_{0},g_{0}}_{\mathrm{T}_{\mathrm{S}_{\mathrm{N}}}}\cup
\Gamma^{u_{0},g_{0}}_{\mathrm{T}_{\mathrm{S}_{\mathrm{D}}}}\cup
\Gamma^{u_{0},g_{0}}_{\mathrm{T}_{\mathrm{S-}}}\cup\Gamma^{u_{0},g_{0}}_{\mathrm{T}_{\mathrm{S+}}},
$$
which is given by
$$
\begin{array}{l}
\Gamma^{u_{0},g_{0}}_{\mathrm{T}_{\mathrm{S}_{\mathrm{N}}}}=\left\{s\in\Gamma_{\mathrm{T}} \mid u_{0}(s)\neq0\right \}=\emptyset,\\
\Gamma^{u_{0},g_{0}}_{\mathrm{T}_{\mathrm{S}_{\mathrm{D}}}}=\left\{(x,y)\in\Gamma \mid -\frac{1}{2}<x<\frac{1}{2}\right\}\cap\Gamma_{\mathrm{T}}=\left\{(\cos\theta,\sin\theta)\in\Gamma \mid \frac{4\pi}{3}<\theta\leq\frac{5\pi}{3}\right\},\\
\Gamma^{u_{0},g_{0}}_{\mathrm{T}_{\mathrm{S-}}}=\left\{(x,y)\in\Gamma \mid x\geq\frac{1}{2}\right\}\cap\Gamma_{\mathrm{T}}=\left\{(\cos\theta,\sin\theta)\in\Gamma \mid -\frac{\pi}{3}<\theta\leq\frac{\pi}{4}\right\},\\
\Gamma^{u_{0},g_{0}}_{\mathrm{T}_{\mathrm{S+}}}=\left\{(x,y)\in\Gamma \mid x\leq-\frac{1}{2}\right\}\cap\Gamma_{\mathrm{T}}=\left\{(\cos\theta,\sin\theta)\in\Gamma \mid \frac{3\pi}{4}<\theta\leq\frac{4\pi}{3}\right\}.
\end{array}
$$
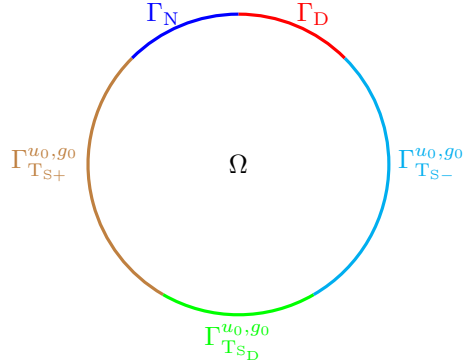
\begin{figure}[ht]
    \centering
\begin{tikzpicture}\label{figure1}
\draw (0,0) node{$\Omega$};
\draw [color=red, very thick] (1.414,1.414) arc (45:90:2);
\draw [color=blue, very thick](0,2) arc(90:135:2);
\draw [color=brown, very thick](-1.414,1.414) arc(135:240:2);
\draw [color=green, very thick](-1,-1.732) arc(240:300:2);
\draw [color=cyan, very thick](1,-1.732) arc(300:405:2);
\draw (1,1.732) [color=red] node[above]{$\Gamma_{\mathrm{D}}$};
\draw (-1,1.732) [color=blue] node[above]{$\Gamma_{\mathrm{N}}$};
\draw (0,-2) [color=green] node[below]{$\Gamma^{u_{0},g_{0}}_{\mathrm{T}_{\mathrm{S}_{\mathrm{D}}}}$};
\draw (2,0) [color=cyan] node[right]{$\Gamma^{u_{0},g_{0}}_{\mathrm{T}_{\mathrm{S-}}}$};
\draw (-2,0) [color=brown] node[left]{$\Gamma^{u_{0},g_{0}}_{\mathrm{T}_{\mathrm{S+}}}$};
\end{tikzpicture}
    \caption{Unit disk $\Omega$ and its boundary $\Gamma=\Gamma_{\mathrm{D}}\cup\Gamma_{\mathrm{N}}\cup\Gamma_{\mathrm{T}}$, with $\Gamma_{\mathrm{T}}=\Gamma^{u_{0},g_{0}}_{\mathrm{T}_{\mathrm{S}_{\mathrm{N}}}}\cup
\Gamma^{u_{0},g_{0}}_{\mathrm{T}_{\mathrm{S}_{\mathrm{D}}}}\cup
\Gamma^{u_{0},g_{0}}_{\mathrm{T}_{\mathrm{S-}}}\cup\Gamma^{u_{0},g_{0}}_{\mathrm{T}_{\mathrm{S+}}}$.}
    \label{figuredeGamma}
\end{figure}\\
Moreover, since $f'_{0}=f$ in $\LL^{2}(\Omega)$, $k'_{0}=\xi$ in $\LL^{2}(\Gamma_{\mathrm{N}})$ and $g'_{0}=1$ in $\LL^{2}(\Gamma_{\mathrm{T}})$,
we are now in a position to compute numerically $u'_{0}$ and $u_{t}$, and then to compare $u_{t}$ with its first-order approximation~$u_{0}+tu'_{0}$ in~$\HH^{1}$-norm for several small values of $t\geq 0$. 
\subsection{Numerical Results}\label{numresults}
Here we present the numerical results obtained for the two-dimensional example described in Section~\ref{presentexample}. Numerical simulations have been made using P2 finite element method and with a discretization of the boundary of 190 points. We concatenate in Table~\ref{tableausansu} some values of~$\left \| u_{t}-u_{0}-tu_{0}' \right \|_{\HH^{1}(\Omega)}$ for several small values of $t\geq0$. Figure~\ref{Cassansu} gives the representation in lo\-garithmic scale of the maps~$t\in\mathbb{R}^{+}\mapsto \left \| u_{t}-u_{0}-tu_{0}' \right \|_{\HH^{1}(\Omega)}\in\R^{+}$ and $t\in\mathbb{R}^{+}\mapsto t^{2}\in\R^{+}$.
Finally, Figure~\ref{fig:figtrescaCasgt} is the illustration of $u_{t}$ and its first-order approximation $u_{0}+tu'_{0}$ for $t=0.1$.

Roughly speaking, we observe from Figure~\ref{Cassansu} that
$$
\left \| \frac{u_{t}-u_{0}}{t}-u_{0}' \right \|_{\HH^{1}(\Omega)}=O(t),
$$
where $O$ stands for the standard Bachmann--Landau notation, which is in accordance with our main result (Theorem~\ref{caractu0derivDNT}).\color{black}
\begin{table}[!ht]
    \center
    \begin{tabular}[b]{|l|c|c|c|c|c|c|c|c|}
    \hline
    Parameter t  & 0.60 & 0.40 & 0.20 & 0.1 & 0.075 & 0.05 & 0.025 & 0.01  \\
    \hline 
    $\left \| u_{t}-u_{0}-tu_{0}' \right \|_{\HH^{1}(\Omega)}$ & 0.6267 & 0.2558 & 0.0590 & 0.0145 & 0.0083 & 0.0042 & 0.0021 & 0.0022 \\
    \hline
    \end{tabular}
    \caption{$\HH^{1}$-norm of the difference between $u_{t}$ and its first-order approximation $u_{0}+tu_{0}'$ for several small values of $t$.}
    \label{tableausansu}
\end{table}
\begin{figure}
    \centering
    \includegraphics[scale=0.5]{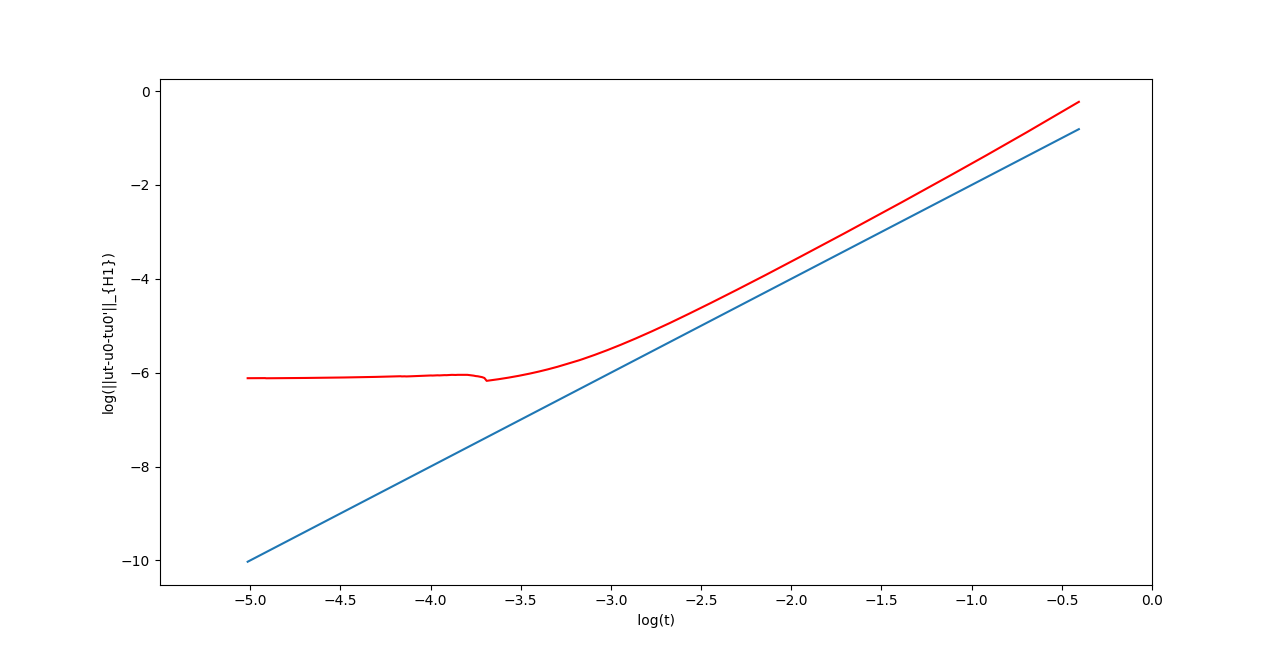}
    \caption{ The representation in logarithmic scale of the map $t\in\R^{+} \mapsto \left \| u_{t}-u_{0}-tu_{0}' \right \|_{\HH^{1}(\Omega)}\in\R^{+}$ (red) and of the map $t\in\R^{+} \mapsto t^{2}\in\R^{+}$ (blue).}
    \label{Cassansu}
\end{figure}
\begin{myRem}
Note that the representation of $\left \| u_{t}-u_{0}-tu_{0}' \right \|_{\HH^{1}(\Omega)}$ with respect to $t$ in logarithmic scale got a threshold for $t\approx0.03$. This is a classical phenomenon due to the numerical approximations we made and the numerical algorithms we used.
\end{myRem}
\begin{figure}[!ht]
    \centering
    \includegraphics[scale=0.35]{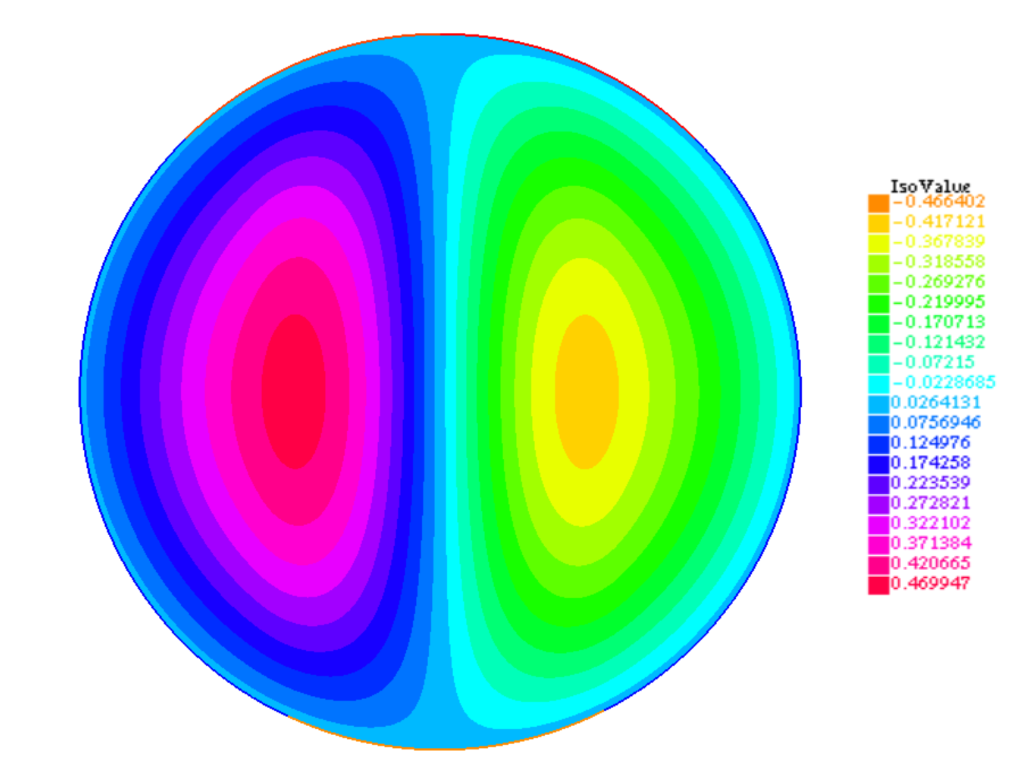}
    \includegraphics[scale=0.57]{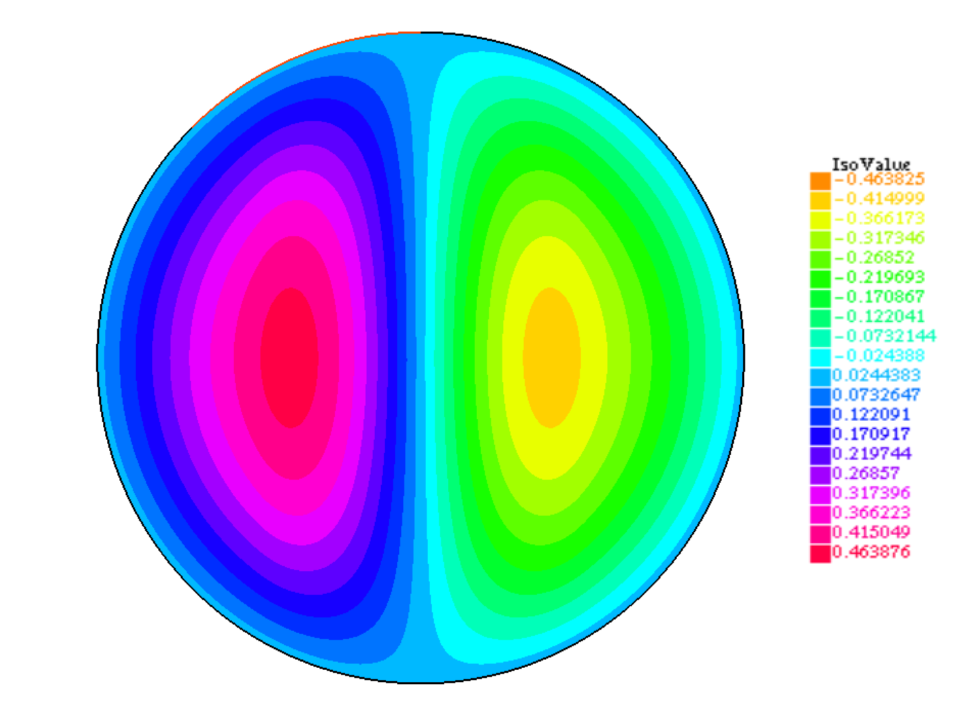}
    \caption{The first figure is the representation of $u_{t}$ and the second its first-order approximation~$u_{0}+~tu_{0}'$ for $t=0.1$.}
    \label{fig:figtrescaCasgt}
\end{figure}

\section{Conclusions}
In this paper we investigated the sensitivity analysis of a scalar mechanical contact problem involving the Tresca's friction law.  This follows the previous paper~\cite{4ABC} where only the right-hand source term was perturbed. In the present work, the friction term associated to the Tresca's friction law was also perturbed which is the main novelty compared to the previous paper. Using tools from convex analysis we proved that the derivative of a parameterized Tresca friction problem is the solution to a problem with Signorini unilateral conditions. This work will be used in order to investigate shape optimization problems involving the Tresca's friction law in a forthcoming article.

\appendix
\section{Sufficient Conditions for the Twice Epi-Differentiability of the Parameterized Tresca Friction Functional}\label{casparticulier}
In this appendix the notations and assumptions introduced in Section~\ref{Mainresult} are preserved. This appendix follows from Remark~\ref{Remarquenotwice}. Our aim here is to prove, in some particular cases which correspond to practical situations, that the parameterized Tresca friction functional $\Phi$ is twice epi-differentiable at $u_{0}$ for $F_{0}-u_{0}\in\partial\Phi(0,\cdot)(u_{0})$, with its second-order epi-derivative given by~\eqref{hypoth1}.
From the characterization of Mosco epi-convergence (see Proposition~\ref{caractMosco}), it is sufficient to prove that, for all~$w\in\HH^{1}_{\mathrm{D}}(\Omega)$, the two conditions
\begin{enumerate}[label=(\roman*)]
    \item for all $(w_{t})_{t>0}\subset\HH^{1}_{\mathrm{D}}(\Omega)$ such that $(w_{t})_{t>0}\rightharpoonup w$ in $\HH^{1}_{\mathrm{D}}(\Omega)$, then 
    $$\mathrm{lim}\inf \Delta_{t}^{2}\Phi(u_{0}|F_{0}-u_{0})(w_{t})\geq \mathrm{I}_{\mathcal{K}_{u_{0},\frac{\partial_{\nn}(F_{0}-u_{0})}{g_{0}}}}(w)+\int_{\Gamma_{\mathrm{T}}}g'_{0}(s)\frac{\partial_{\nn}(F_{0}-u_{0})(s)}{g_{0}(s)}w(s)\mathrm{d}s ;
    $$\label{premiercondit}
    \item there exists $(w_{t})_{t>0}\subset\HH^{1}_{\mathrm{D}}(\Omega)$ such that $(w_{t})_{t>0}\rightarrow w$ in $\HH^{1}_{\mathrm{D}}(\Omega)$ and 
    $$
    \mathrm{lim}\sup \Delta_{t}^{2}\Phi(u_{0}|F_{0}-u_{0})(w_{t})\leq \mathrm{I}_{\mathcal{K}_{u_{0},\frac{\partial_{\nn}(F_{0}-u_{0})}{g_{0}}}}(w)+\int_{\Gamma_{\mathrm{T}}}g'_{0}(s)\frac{\partial_{\nn}(F_{0}-u_{0})(s)}{g_{0}(s)}w(s)\mathrm{d}s ;
    $$
\end{enumerate}
are satisfied.

The condition~\ref{premiercondit} is always satisfied. Indeed, from Proposition~\ref{epidiffoffunctionG}, this condition can be rewritten as
\begin{multline*}
    \mathrm{lim}\inf\int_{\Gamma_{\mathrm{T}}}\Delta_{t}^{2}G(s)(u(s)|\partial_{\nn}(F_{0}-u_{0})(s))(w_{t}(s))\mathrm{d}s\\ \geq \int_{\Gamma_{\mathrm{T}}}\mathrm{D}_{e}^{2}G(s)(u_{0}(s)|\partial_{\nn}(F_{0}-u_{0})(s))(w(s))\mathrm{d}s,
\end{multline*}
which is true thanks to the dense and compact embedding~$\HH^{1}(\Omega)\hookdoubleheadrightarrow \mathrm{L}^{2}(\Gamma)$, to the twice epi-differentiability of the function $G(s)$ for almost all $s\in\Gamma_{\mathrm{T}}$ (see Proposition~\ref{épidiffgabs}) and to the classical Fatou's lemma (see, e.g.,~\cite[Lemma 4.1 p.90]{BREZ}).

The condition (ii) is obviously satisfied if $w\notin\mathcal{K}_{u_{0},\frac{\partial_{\nn}(F_{0}-u_{0})}{g_{0}}}$. Thus, one has only to prove the following assertion:
\begin{enumerate}[label=({\roman*$'$})]
\setcounter{enumi}{1}
\item for all $w\in\mathcal{K}_{u_{0},\frac{\partial_{\nn}(F_{0}-u_{0})}{g_{0}}}$, there exists $(w_{t})_{t>0}\subset\HH^{1}_{\mathrm{D}}(\Omega)$ such that $(w_{t})_{t>0}\rightarrow w$ in $\HH^{1}_{\mathrm{D}}(\Omega)$~and
$$
\mathrm{lim}\sup \Delta_{t}^{2}\Phi(u_{0}|F_{0}-u_{0})(w_{t})\leq \int_{\Gamma_{\mathrm{T}}}g'_{0}(s)\frac{\partial_{\nn}(F_{0}-u_{0})(s)}{g_{0}(s)}w(s)\mathrm{d}s.
$$\label{condrestantefortwice}
\end{enumerate}
Unfortunately we are not able to prove this assertion in a general setting yet, that is without any additional assumption on~$u_{0}$ and on~$\Gamma$, and in any dimension $d\geq1$. Nevertheless, in this appendix, we prove this assertion in some particular cases which correspond to practical situations, providing sufficient conditions. In particular, in the next sections, we consider the additional assumption 
\begin{enumerate}[label=(\Alph*)]
        \item \begin{center}the map $t\in\mathbb{R}^{+}\mapsto g_{t}\in \LL^{2}(\Gamma_{\mathrm{T}})$ is differentiable at $t=0$\label{assumptionuseful}.
         \end{center}
\end{enumerate} 

\subsection{First Example of Sufficient Condition: $u=0$ almost everywhere on~$\Gamma_{\mathrm{T}}$}
In this first example, we assume that $u_{0}=0$ almost everywhere on $\Gamma_{\mathrm{T}}$, therefore~$\Gamma^{u_{0},g_{0}}_{\mathrm{T}_{\mathrm{S}_{\mathrm{N}}}}$ has a null measure. Let $w\in\mathcal{K}_{u_{0},\frac{\partial_{\nn}(F_{0}-u_{0})}{g_{0}}}$. Then, taking the sequence $w_{t}=w$ for all $t>0$, one gets
\begin{multline*}
\Delta_{t}^{2}\Phi(u_{0}|F_{0}-u_{0})(w)=\\\int_{\Gamma^{u_{0},g_{0}}_{\mathrm{T}_{\mathrm{S+}}}\cup\Gamma^{u_{ 0},g_{0}}_{\mathrm{T}_{\mathrm{S-}}}}\frac{g_{t}(s)|u_{0}(s)+t w(s)|-g_{t}(s)|u_{0}(s)|+t\partial_{\nn}(F_{0}-u_{0})(s)w(s)}{t^{2}}\mathrm{d}s \\
= \int_{\Gamma^{u_{0},g_{0}}_{\mathrm{T}_{\mathrm{S+}}}}\frac{g_{t}(s)-g_{0}(s)}{t}w(s)\mathrm{d}s-\int_{\Gamma^{u_{0},g_{0}}_{\mathrm{T}_{\mathrm{S-}}}}\frac{g_{t}(s)-g_{0}(s)}{t}w(s)\mathrm{d}s \\
\longrightarrow \int_{\Gamma_{\mathrm{T}}}g'_{0}(s)\frac{\partial_{\nn}(F_{0}-u_{0})(s)}{g_{0}(s)}w(s)\mathrm{d}s,
\end{multline*}
when $t\rightarrow 0^{+}$ from Assumption~\ref{assumptionuseful}. Therefore Condition~\ref{condrestantefortwice} is satisfied.
\subsection{Second Example of Sufficient Condition: Truncature}
In this second example, we introduce two disjoint subsets of $\Gamma_{\mathrm{T}}$ given by
$$
\Gamma^{u_{0},g_{0}}_{\mathrm{T}_{\mathrm{S}_{\mathrm{N}+}}}:=\left\{s\in\Gamma_{\mathrm{T}} \mid u_{0}(s)>0\right\}
\qquad \mbox{ and } \qquad 
\Gamma^{u_{0},g_{0}}_{\mathrm{T}_{\mathrm{S}_{\mathrm{N}-}}}:=\left\{s\in\Gamma_{\mathrm{T}} \mid  u_{0}(s)<0\right\}.
$$
Hence it follows that $\Gamma^{u_{0},g_{0}}_{\mathrm{T}_{\mathrm{S}_{\mathrm{N}}}}=\Gamma^{u_{0},g_{0}}_{\mathrm{T}_{\mathrm{S}_{\mathrm{N}+}}}\cup\Gamma^{u_{0},g_{0}}_{\mathrm{T}_{\mathrm{S}_{\mathrm{N}-}}}$, $\partial_{\nn}u_{0}=-g_{0}$ almost everywhere on $\Gamma^{u_{0},g_{0}}_{\mathrm{T}_{\mathrm{S}_{\mathrm{N}+}}}$ and that~$\partial_{\nn}u_{0}=g_{0}$ almost everywhere on $\Gamma^{u_{0},g_{0}}_{\mathrm{T}_{\mathrm{S}_{\mathrm{N}-}}}$. Now let us assume that there exists $C>0$ such that~$|u_{0}|\geq C$ on $\Gamma^{u_{0},g_{0}}_{\mathrm{T}_{\mathrm{S}_{\mathrm{N}+}}}\cup\Gamma^{u_{0},g_{0}}_{\mathrm{T}_{\mathrm{S}_{\mathrm{N}-}}}$. Let us consider $w\in\mathcal{K}_{u_{0},\frac{\partial_{\nn}(F_{0}-u_{0})}{g_{0}}}$ and the truncature $w_{t}\in\HH^{1}_{\mathrm{D}}(\Omega)$ of $w$ defined by
$$
w_{t}(x):=
\left\{
\begin{array}{ccll}
\frac{1}{\sqrt{t}} 	&     & \text{ if } w(x)\geq\frac{1}{\sqrt{t}} , \\
w(x) 	&   & \text{ if } |w(x)| \leq\frac{1}{\sqrt{t}}, \\
-\frac{1}{\sqrt{t}} 	&     & \text{ if } w(x)\leq-\frac{1}{\sqrt{t}},
\end{array}
\right.
$$
for almost all $x\in\Omega$ and for all $t>0$. One deduces from Marcus-Mizel theorem (see~\cite{MarcusMizel}) that~$w_{t}\rightarrow w$ in $\HH^{1}_{\mathrm{D}}(\Omega)$ when $t\rightarrow 0^{+}$. Moreover, for all $t\leq C^{2}$, one gets
\begin{multline*}
\Delta_{t}^{2}\Phi(u_{0}|F_{0}-u_{0})(w_{t})=\\\int_{\Gamma^{{u_{0}},g_{0}}_{\mathrm{T}_{\mathrm{S+}}}\cup\Gamma^{u_{0},g_{0}}_{\mathrm{T}_{\mathrm{S}_{\mathrm{N}+}}}}\frac{g_{t}(s)-g_{0}(s)}{t}w_{t}(s)\mathrm{d}s-\int_{\Gamma^{u_{0},g_{0}}_{\mathrm{T}_{\mathrm{S-}}}\cup\Gamma^{u_{0},g_{0}}_{\mathrm{T}_{\mathrm{S}_{\mathrm{N}-}}}}\frac{g_{t}(s)-g_{0}(s)}{t}w_{t}(s)\mathrm{d}s\\\longrightarrow \int_{\Gamma_{\mathrm{T}}}g'_{0}(s)\frac{\partial_{\nn}(F_{0}-u_{0})(s)}{g_{0}(s)}w(s)\mathrm{d}s,
\end{multline*}
when $t\rightarrow0^{+}$ from Assumption~\ref{assumptionuseful}, therefore Condition~\ref{condrestantefortwice} is satisfied.
\subsection{Third Example of Sufficient Condition: Truncature and Dilatation}
In this third example, we take $d=2$ and $\Gamma_{\mathrm{N}}=\emptyset$, and we assume that $u_{0}$ and $\partial_{\nn}u_{0}$ are conti\-nuous on $\Gamma$, and that $\Gamma$ is diffeomorphic to the circle $\mathrm{S}^{1}:=\left\{(x,y)\in\R^{2} \mid x^{2}+y^{2}=1\right\}$. From this last assumption, for simplicity, we assume in the sequel that~$\Gamma=\mathrm{S}^{1}$.
In what follows, the next hypotheses are only useful to simplify the computations. Let us assume that $\Gamma_{\mathrm{T}}=\Gamma^{u_{0},g_{0}}_{\mathrm{T}_{\mathrm{S+}}}\cup\Gamma^{u_{0},g_{0}}_{\mathrm{T}_{\mathrm{S}_{\mathrm{N}+}}}$ (in this particular case, the hypothesis on the continuity of~$\partial_{\nn}u_{0}$ is useless, see Remark~\ref{continuitedednuo}) where~$\Gamma^{u_{0},g_{0}}_{\mathrm{T}_{\mathrm{S}_{\mathrm{N}+}}}$ has already been defined in the previous example, and with the following parameterizations
$$
\displaystyle\Gamma^{u_{0},g_{0}}_{\mathrm{T}_{\mathrm{N}+}}=\left\{ \left(\cos{\theta},\sin{\theta}\right)\in\Gamma \mid \theta\in\left] \gamma_{1},\gamma_{2}\right[ \right\},
$$
$$
\displaystyle\Gamma^{u_{0},g_{0}}_{\mathrm{T}_{\mathrm{S+}}}=\left\{ \left(\cos{\theta},\sin{\theta}\right)\in\Gamma \mid \theta\in\left[ \xi_{1},\gamma_{1}\right]\cup\left[ \gamma_{2},\xi_{2}\right] \right\},
$$
such that $-\pi\leq\xi_{1}<\gamma_{1}<\gamma_{2}<\xi_{2}\leq\pi$ (see Figure~\ref{fig:figuredeGamma}). From the continuity of $u_{0}$, there exists~$c>0$ such that $u_{0}\geq c$ on the set $\{(\cos{\theta},\sin{\theta})\in\Gamma\text{, } \theta\in\left[ \chi_{1},\chi_{2}\right] \}\subset\Gamma^{u_{0},g_{0}}_{\mathrm{T}_{\mathrm{N}+}}$, with $\gamma_{1}<\chi_{1}<\chi_{2}<\gamma_{2}$. Let us consider $\omega_{1}\in\left]\xi_{1},\gamma_{1}\right[$, $\omega_{2}\in\left]\gamma_{2},\xi_{2}\right[$, and also $\alpha_{t}$, $\beta_{t}$ defined, for $t>0$ such that $\sqrt{t}\leq c$, by
$$
\alpha_{t}:=\inf{\left\{ \alpha\in\left[ \gamma_{1},\chi_{1}\right] \mid \forall\theta\in\left[\alpha,\chi_{1}\right]\text{, }u_{0}(\cos{\theta},\sin{\theta})\geq\sqrt{t} \right\}},
$$
$$
\beta_{t}:=\inf{\left\{ \beta\in\left[ \chi_{2},\gamma_{2}\right] \mid \forall\theta\in\left[\chi_{2},\beta\right]\text{, }u_{0}(\cos{\theta},\sin{\theta})\geq\sqrt{t} \right\}}.
$$
From the continuity of $u_{0}$, ones deduces that $\alpha_{t}\rightarrow\gamma_{1}$ and $\beta_{t}\rightarrow\gamma_{2}$ when $t\rightarrow0^{+}$.
\begin{figure}[ht]
    \centering
\begin{tikzpicture}\label{figure1}
\draw (0,0) node{$\Omega$} ;
\draw [color=red, very thick] (1.932,-0.518) arc (345:450:2) ;
\draw [color=blue, very thick](0,2) arc(90:135:2);
\draw [color=green, very thick](-1.414,1.414) arc(135:280:2);
\draw [color=blue, very thick](0.347,-1.97) arc(280:345:2);
\draw (1.8,1.8) [color=red] node[above]{$\Gamma^{u_{0},g_{0}}_{\mathrm{T}_{\mathrm{S}_{\mathrm{N}+}}}$};
\draw (2,0) [color=black] node{$\times$};
\draw (2,0) [color=red] node[right] {$\alpha_{\tau}$};
\draw (0.518,1.932) [color=black] node{$\times$};
\draw (0.518,1.95) [color=red] node[above] {$\beta_{\tau}$};
\draw (1.414,1.414) [color=black] node{$\times$};
\draw (1.414,1.414) [color=red] node[right] {$\chi_{2}$};
\draw (1.879,0.684) [color=black] node{$\times$};
\draw (1.879,0.684) [color=red] node[right] {$\chi_{1}$};
\draw (0,2) [color=black] node{$\times$};
\draw (0,2.1) [color=red] node[above] {$\gamma_{2}$};
\draw (-1.5,1.9) [color=blue] node[above]{$\Gamma^{u_{0},g_{0}}_{\mathrm{T}_{\mathrm{S+}}}$};
\draw (1.932,-0.518) [color=black] node{$\times$};
\draw (2,-0.518) [color=red] node[right] {$\gamma_{1}$};
\draw (0.347,-1.97) [color=black] node{$\times$};
\draw (0.347,-1.97) [color=blue] node[below] {$\xi_{1}$};
\draw (-1.414,1.414) [color=black] node{$\times$};
\draw (-1.414,1.5) [color=blue] node[left] {$\xi_{2}$};
\draw (-0.684,1.879) [color=black] node{$\times$};
\draw (-0.684,1.9) [color=blue] node[above] {$\omega_{2}$};
\draw (1.286,-1.532) [color=black] node{$\times$};
\draw (1.5,-1.6) [color=blue] node[below] {$\omega_{1}$};
\draw (-1.286,-1.53) [color=green] node[left]{$\Gamma_{\mathrm{N}}\cup\Gamma_{\mathrm{D}}$};
\end{tikzpicture}
    \caption{Illustration of the boundary $\Gamma$}
    \label{fig:figuredeGamma}
\end{figure}
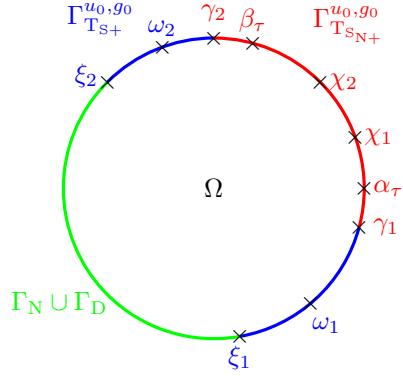

Let $w\in\mathcal{K}_{u_{0},\frac{\partial_{\nn}(F_{0}-u_{0})}{g_{0}}}$, and let $y_{t}\in\HH^{1}_{\mathrm{D}}(\Omega)$ be the truncature of $w$ given by
$$
y_{t}(x):=
\left\{
\begin{array}{ccll}
\frac{1}{\sqrt{t}} 	&     & \text{ if } w(x)\geq\frac{1}{\sqrt{t}} , \\
w(x) 	&   & \text{ if } |w(x)| \leq\frac{1}{\sqrt{t}}, \\
-\frac{1}{\sqrt{t}} 	&     & \text{ if } w(x)\leq-\frac{1}{\sqrt{t}},
\end{array}
\right.
$$
for almost all $x\in\Omega$ and for all $t>0$. As in the previous section, one gets $y_{t}\rightarrow w$ in $\HH^{1}_{\mathrm{D}}(\Omega)$, and thus~$y_{t|\Gamma}\rightarrow w_{|\Gamma}$ in $\HH^{1/2}(\Gamma)$ when~$t\rightarrow0^{+}$.
Let us consider, for $t>0$ sufficiently small, the dilatation~$z_{t}:=y_{t|\Gamma} \circ d_{t}$ of $y_{t|\Gamma}$, with~$d_{t}$ given by
$$
\fonction{d_{t}}{\Gamma}{\Gamma}{(x_{1},x_{2})}{\left\{
\begin{array}{ccll}
\left(x_{1},x_{2}\right) &  & \text{ if }\left(x_{1},x_{2}\right) \in\Gamma_{\mathrm{N}}\cup\Gamma_{\mathrm{D}} , \\
\left(x_{1},x_{2}\right) &  & \text{ if }\left(x_{1},x_{2}\right) \in\Gamma^{u_{0},g_{0}}_{\mathrm{T}_{\mathrm{S+}}}\textbackslash\left\{\left(\cos{\theta},\sin\theta\right)\text{, } \theta\in\left[\omega_{1},\omega_{2}\right]\right\}, \\
d^{\omega_{1},\alpha_{t}}(x_{1},x_{2}) 	&   & \text{ if } 2\arctan{\left(\frac{x_{2}}{x_{1}+1}\right)}\in\left[\omega_{1},\alpha_{t}\right], \\
d^{\alpha_{t},\chi_{1}}(x_{1},x_{2}) 	&   & \text{ if } 2\arctan{\left(\frac{x_{2}}{x_{1}+1}\right)}\in\left[\alpha_{t},\chi_{1}\right], \\
\left(x_{1},x_{2}\right) 	&     & \text{ if } 2\arctan{\left(\frac{x_{2}}{x_{1}+1}\right)}\in\left[\chi_{1},\chi_{2}\right], \\
d^{\chi_{2},\beta_{t}}(x_{1},x_{2}) 	&   & \text{ if } 2\arctan{\left(\frac{x_{2}}{x_{1}+1}\right)}\in\left[\chi_{2},\beta_{t}\right], \\
d^{\beta_{t},\omega_{2}}(x_{1},x_{2}) 	&   & \text{ if } 2\arctan{\left(\frac{x_{2}}{x_{1}+1}\right)}\in\left[\beta_{t}, \omega_{2}\right], \\
\end{array}
\right.}
$$
where 
$$
d^{\omega_{1},\alpha_{t}}(x_{1},x_{2})=\left(\cos{\theta^{\omega_{1},\alpha_{t}}}, \sin {\theta^{\omega_{1},\alpha_{t}}}\right), \text{ with } \theta^{\omega_{1},\alpha_{t}}=\frac{\left(\gamma_{1}-\omega_{1}\right)2\arctan{\left(\frac{x_{2}}{x_{1}+1}\right)}+w_{1}\left(\alpha_{t}-\gamma_{1}\right)}{\alpha_{t}-\omega_{1}},
$$
$$
d^{\alpha_{t},\chi_{1}}(x_{1},x_{2})=\left(\cos{\theta^{\alpha_{t},\chi_{1}}}, \sin {\theta^{\alpha_{t},\chi_{1}}}\right), \text{ with } \theta^{\alpha_{t},\chi_{1}}=\frac{\left(\chi_{1}-\gamma_{1}\right)2\arctan{\left(\frac{x_{2}}{x_{1}+1}\right)}+\chi_{1}\left(\gamma_{1}-\alpha_{t}\right)}{\chi_{1}-\alpha_{t}},
$$
$$
d^{\chi_{2},\beta_{t}}(x_{1},x_{2})=\left(\cos{\theta^{\chi_{2},\beta_{t}}}, \sin {\theta^{\chi_{2},\beta_{t}}}\right), \text{ with } \theta^{\chi_{2},\beta_{t}}=\frac{\left(\gamma_{2}-\chi_{2}\right)2\arctan{\left(\frac{x_{2}}{x_{1}+1}\right)}+\chi_{2}\left(\beta_{t}-\gamma_{2}\right)}{\beta_{t}-\chi_{2}},
$$
$$
d^{\beta_{t},\omega_{2}}(x_{1},x_{2})=\left(\cos{\theta^{\beta_{t},\omega_{2}}}, \sin {\theta^{\beta_{t},\omega_{2}}}\right), \text{ with } \theta^{\beta_{t},\omega_{2}}=\frac{\left(\omega_{2}-\gamma_{2}\right)2\arctan{\left(\frac{x_{2}}{x_{1}+1}\right)}+\omega_{2}\left(\gamma_{2}-\beta_{t}\right)}{\omega_{2}-\beta_{t}}.
$$
Note that, since $-\pi\leq\xi_{1}<\omega_{1}<\omega_{2}<\xi_{2}\leq\pi$ (see Remark~\ref{anglepi}), then $d_{t}$ is a well-defined bijective Lipschitz continuous map, and its inverse is also a bijective Lipschitz continuous map. Thus it follows that $z_{t}\in\HH^{1/2}(\Gamma)$ and also $z_{t}\rightarrow w_{|\Gamma}$ in $\HH^{1/2}(\Gamma)$ when $t\rightarrow0^{+}$. Then, for $t>0$ sufficiently small, we denote by $w_{t}\in\HH^{1}_{\mathrm{D}}(\Omega)$ a lift of $z_{t}\in\HH^{1/2}(\Gamma)$, such that $w_{t}\rightarrow w$ in $\HH^{1}_{\mathrm{D}}(\Omega)$ when~$t\rightarrow0^{+}$. Therefore, by denoting 
$$
m_{t}(s)=\frac{g_{t}(s)|u_{0}(s)+t w_{t}(s)|-g_{t}(s)|u_{0}(s)|-t\partial_{\nn}(F_{0}-u_{0})(s)w_{t}(s)}{t^{2}},
$$
for $t>0$ sufficiently small and for almost all $s\in\Gamma_{\mathrm{T}}$, it follows that
\begin{multline*}
\Delta_{t}^{2}\Phi(u_{0}|F_{0}-u_{0})(w_{t}) = \int_{\left\{ \left(\cos{\theta},\sin{\theta}\right)\text{, } \theta\in\left[\xi_{1},\omega_{1}\right]\right\}}m_{t}(s)\mathrm{d}s \\
+\int_{\left\{ \left(\cos{\theta},\sin{\theta}\right)\text{, } \theta\in\left[\omega_{1},\alpha_{t}\right]\right\}}m_{t}(s)\mathrm{d}s
+\int_{\left\{ \left(\cos{\theta},\sin{\theta}\right)\text{, } \theta\in\left[\alpha_{t},\chi_{1}\right]\right\}}m_{t}(s)\mathrm{d}s \\
+\int_{\left\{ \left(\cos{\theta},\sin{\theta}\right)\text{, } \theta\in\left[\chi_{1},\chi_{2}\right]\right\}}m_{t}(s)\mathrm{d}s 
+\int_{\left\{ \left(\cos{\theta},\sin{\theta}\right)\text{, } \theta\in\left[\chi_{2},\beta_{t}\right]\right\}}m_{t}(s)\mathrm{d}s \\
+\int_{\left\{ \left(\cos{\theta},\sin{\theta}\right)\text{, } \theta\in\left[\beta_{t},\omega_{2}\right]\right\}}m_{t}(s)\mathrm{d}s+\int_{\left\{ \left(\cos{\theta},\sin{\theta}\right)\text{, } \theta\in\left[\omega_{2},\xi_{2}\right]\right\}}m_{t}(s)\mathrm{d}s.
\end{multline*}
Then, from the definition of $d_{t}$ and Assumption~\ref{assumptionuseful}, one deduces that
$$
\Delta_{t}^{2}\Phi(u_{0}|F_{0}-u_{0})(w_{t})\longrightarrow \int_{\Gamma_{\mathrm{T}}}g'_{0}(s)\frac{\partial_{\nn}(F_{0}-u_{0})(s)}{g_{0}(s)}w(s)\mathrm{d}s,
$$
when $t\rightarrow0^{+}$, and thus Condition~\ref{condrestantefortwice} is satisfied. 
\begin{myRem}\label{continuitedednuo}
In the case where $\Gamma_{\mathrm{T}}=\Gamma^{u_{0},g_{0}}_{\mathrm{T}_{\mathrm{S+}}}\cup\Gamma^{u_{0},g_{0}}_{\mathrm{T}_{\mathrm{S}_{\mathrm{N}+}}}$, the hypothesis $\partial_{\nn}u_{0}$ continuous on $\Gamma$ is useless. Nevertheless, in the general case $\Gamma_{\mathrm{T}}=\Gamma^{u_{0},g_{0}}_{\mathrm{T}_{\mathrm{S}_{\mathrm{N}+}}}\cup\Gamma^{u_{0},g_{0}}_{\mathrm{T}_{\mathrm{S}_{\mathrm{N}-}}}\cup
\Gamma^{u_{0},g_{0}}_{\mathrm{T}_{\mathrm{S}_{\mathrm{D}}}}\cup\Gamma^{u_{0},g_{0}}_{\mathrm{T}_{\mathrm{S-}}}\cup\Gamma^{u_{0},g_{0}}_{\mathrm{T}_{\mathrm{S+}}}$, the hypotheses~$u_{0}$ and~$\partial_{\nn}u_{0}$ continuous on $\Gamma$ is sufficient to get the twice epi-differentiability of the parameterized Tresca friction functional: a part of $\Gamma^{u_{0},g_{0}}_{\mathrm{T}_{\mathrm{S-}}}$ (resp.\ $\Gamma^{u_{0},g_{0}}_{\mathrm{T}_{\mathrm{S+}}}$, resp.\ $\Gamma^{u_{0},g_{0}}_{\mathrm{T}_{\mathrm{S}_{\mathrm{N}-}}}$)  is never side to side with a part of $\Gamma^{u_{0},g_{0}}_{\mathrm{T}_{\mathrm{S}_{\mathrm{N}+}}}$ (resp.\ $\Gamma^{u_{0},g_{0}}_{\mathrm{T}_{\mathrm{S}_{\mathrm{N}-}}}$, resp.\ $\Gamma^{u_{0},g_{0}}_{\mathrm{T}_{\mathrm{S}_{\mathrm{N}+}}}$), and thus, using an appropriate dilatation, one can obtain the same result.
\end{myRem}
\begin{myRem}\label{anglepi}
The hypothesis on the angles
$$-\pi\leq\xi_{1}<\omega_{1}<\gamma_{1}<\gamma_{2}<\omega_{2}<\xi_{2}\leq\pi,
$$ avoids the problem of the definition of $d_{t}$ for the point $(x_{1},x_{2})=(-1,0)$. But, in a more general case, since $\Gamma_{\mathrm{D}}$ has a positive measure, it is always possible to translate the angles in order to overcome this difficulty and get a well-defined dilatation $d_{t}$.
\end{myRem}
\begin{myRem}
 The assumption $\Gamma_{\mathrm{N}}=\emptyset$ can be replaced by the assumption that $\Gamma_{\mathrm{N}}$ is never side to side with $\Gamma^{u_{0},g_{0}}_{\mathrm{T}_{\mathrm{S}_{\mathrm{N}+}}}$  and $\Gamma^{u_{0},g_{0}}_{\mathrm{T}_{\mathrm{S}_{\mathrm{N}-}}}$. Without one of those assumptions, the dilatation may not work.
\end{myRem}

\bibliographystyle{abbrv}
\bibliography{biblio}
\end{document}